\documentclass[11pt]{article}

\usepackage{url}
\usepackage{amssymb,enumerate}
\usepackage{amsmath,amsfonts}
\usepackage{amsthm}
\usepackage{latexsym}
\usepackage{mathrsfs}
\usepackage{color}
\usepackage{microtype}
\usepackage{geometry}

\usepackage{algorithm}
\usepackage{algorithmicx}
\usepackage{algpseudocode}

\usepackage{enumitem}

\usepackage{multirow}
\allowdisplaybreaks[4]

\usepackage{amsmath,amsfonts,bm}

\def\1{\bm{1}}

\def\rmX{{\mathbf{X}}}

\DeclareMathAlphabet{\mathsfit}{\encodingdefault}{\sfdefault}{m}{sl}
\SetMathAlphabet{\mathsfit}{bold}{\encodingdefault}{\sfdefault}{bx}{n}

\newcommand{\E}{\mathbb{E}}

\newcommand{\R}{\mathbb{R}}

\DeclareMathOperator*{\argmin}{arg\,min}

\newcommand{\MBspace}{\ensuremath{\R^{{\bf m}\odot{\bf n}}}}

\theoremstyle{plain}
\newtheorem{theorem}{Theorem}

\newtheorem{lemma}{Lemma}

\newtheorem{assumption}{Assumption}

\newtheorem{remark}{Remark}

\usepackage{graphicx, color}
\usepackage{algorithm, algpseudocode} 
\usepackage{mathrsfs} 

\usepackage{lipsum}

\usepackage[colorlinks=true,
            linkcolor=blue,
            citecolor=blue,
            urlcolor=blue]{hyperref}

\title{Schattor: Schatten-family methods for deep learning optimization}
\author{Bohao Ma\thanks{School of Data Science, The Chinese University of Hong Kong (Shenzhen), Shenzhen, Guangdong, China (email: {\tt bohaoma@link.cuhk.edu.cn}).}
\and
Junyu Zhang\thanks{Department of Industrial Systems Engineering and Management, National University of Singapore, Singapore (email: {\tt junyuz@nus.edu.sg}). The work of Junyu Zhang was partially supported by the Singapore Ministry of Education Academic Research Fund Tier 2 (MOE-T2EP20125-0007).}
\and
Chuan He\thanks{Department of Mathematics, Link\"oping University, Sweden (email: {\tt chuan.he@liu.se}). The work of Chuan He was partially supported by the Wallenberg AI, Autonomous Systems and Software Program (WASP) funded by the Knut and Alice Wallenberg Foundation. Corresponding author.}
}

\date{
}

\begin{document}
\maketitle

    
\begin{abstract}

Modern deep learning optimization features heterogeneous parameter structures, noisy gradients, and highly nonconvex landscapes, posing significant challenges for both algorithm design and theoretical analysis. Motivated by the limitations of SGD and the success of adaptive optimizers, we propose {\it Schattor}, a family of adaptive first-order methods based on Schatten norms. Schattor unifies SGD and the recently proposed matrix-variate adaptive optimizer Muon within a single Schatten-norm-based framework. We establish dimension-free stationarity guarantees for methods in the Schattor family for stochastic matrix optimization problems via a novel matrix martingale moment bound. We also develop multi-block extensions that adaptively balance block-wise optimization progress and prove dimension-free stationarity guarantees in this more general setting.

\end{abstract}


\smallskip


\section{Introduction}

Training neural networks is one of the most important optimization problems in modern machine learning practice \cite{brown2020language,lecun2015deep,vaswani2017attention}, yet it remains a highly challenging task due to the large scale and complex structure of the  underlying problems. The most basic method is SGD, which enables scalable training by replacing full gradients with mini-batch stochastic estimates \cite{bottou2010large}. Despite its relatively mature theoretical understanding and extensive study, SGD-based methods are often insufficient for training modern neural networks effectively, including widely used architectures such as transformers; see, e.g., \cite{zhang2020adaptive,zhang2024transformers}. 

This limitation has inspired interest in adaptive optimization methods; see, e.g., \cite{bernstein2018signsgd,gupta2018shampoo,jordan2024muon,kingma2014adam,mishchenko2024prodigy}, which rely on more balanced search directions that better account for the heterogeneity across parameters, layers, and gradient scales. In particular, Adam \cite{kingma2014adam}, one of the most widely used adaptive optimizers for neural network training, performs coordinate-wise adaptive updates by rescaling first-order moment estimates of stochastic gradients using estimates of second moments. Notably, despite its strong empirical success, Adam does not explicitly exploit any particular neural network architecture or structural property. Beyond Adam, there has recently been growing interest in alternative principles for adaptive deep learning optimization. For example, signSGD \cite{bernstein2018signsgd} adopts a strategy that normalizes coordinate-wise magnitudes by retaining only the signs of the gradient entries. More recently, Muon \cite{jordan2024muon} can be interpreted as a matrix-variate analogue of this idea, extending it from coordinate-wise normalization to spectral normalization. Instead of enforcing identical magnitudes in the original coordinate space, Muon operates on matrix-valued gradients and equalizes the singular values of the update direction, thereby mitigating scale heterogeneity in matrix-valued gradients. A more complete review of the relevant developments is deferred to Appendix \ref{apx:related-work}.

Moving from high-level algorithmic intuition to rigorous analysis, norm-based perspectives provide a natural framework for understanding adaptive optimization methods. In particular, the ideas behind signSGD and Muon can be interpreted as constructing updates using the $\ell_\infty$-norm and the spectral norm (two extreme choices from the $\ell_p$-norm and Schatten-$p$ norm families). In contrast, traditional SGD can be interpreted as relying on the standard Euclidean norm. Building on this norm-based perspective, we aim to unify first-order methods induced by the Schatten family of norms within a single framework, which we call {\it Schattor}. This framework enables the exploration of less extreme norm choices than the Schatten-$\infty$ geometry underlying Muon, while retaining stronger mechanisms for mitigating gradient heterogeneity than SGD-type methods, which correspond to the Schatten-$2$, or Frobenius, geometry in the matrix setting.

To introduce \emph{Schattor}, we begin by considering the single-block matrix optimization problem:
\begin{align*}
    \min_{X\in\R^{m\times n}} f(X),
\end{align*}
where $f:\R^{m\times n}\to\R$ is a continuously differentiable objective function. 
We next highlight the distinctive features of Schattor:


\begin{enumerate}[leftmargin = .6cm]
    \item[{(i)}] \textbf{Two update rules.} Given a gradient estimate $M_k$, the $S_p$-norm-based ball oracle computes the update direction by minimizing the ball-constrained linear approximation:
    \begin{align*}
        Z_k\in\argmin_{\|Z\|_{S_p}\le 1}\ \langle M_k,Z\rangle,
        \qquad
        X_{k+1}=X_k+\eta Z_k .
    \end{align*}
    The $S_p$-norm-based regularization oracle instead minimizes the quadratically regularized linear approximation
    \begin{align*}
    Z_k\in\argmin_{Z\in\R^{m\times n}}
        \left\{\langle M_k,Z\rangle+\frac{1}{2}\|Z\|_{S_p}^2\right\},
        \qquad
        X_{k+1}=X_k+\eta Z_k.
    \end{align*}
    These two primitives define the local update rules studied throughout the paper. The $S_p$-ball oracle is closely related to linear-minimization-oracle training \cite{pethick2025training}; it also captures the original spectral-norm view of Muon-type orthogonalized updates
    \cite{jordan2024muon,shulgin2025beyond}. The regularization oracle can be interpreted as an update rule derived from the majorization-minimization principle. Compared with the $S_p$-ball oracle, it yields the same update direction but introduces an additional rescaling based on the size of $M_k$; see Lemma~\ref{lem:single-block-oracle} for the derivation and, e.g., \cite{lau2025polargrad} for a regularization-oracle-based variant of Muon. 

    \item[{(ii)}] \textbf{Continuum of Schatten geometries.} Varying $p$ controls the geometry induced by the oracle, thereby shaping the spectral distribution of the resulting update for a given gradient estimate $M_k$. When $p=2$, the update aligns with the estimate of the negative gradient $-M_k$. When $p=\infty$, the update has all singular values equal to one, corresponding to an orthogonalized direction. For intermediate values $p\in(2,\infty)$, the singular-value distribution of the update interpolates between the original singular values of $M_k$ and all singular values equal to one.

    For optimization problems over $\R^n$, oracle-based methods using $\ell_p$-norm geometries are naturally included in our framework. Indeed, by restricting the square matrix variable to the space of diagonal matrices, the Schatten-$p$ norm reduces to the $\ell_p$-norm of the diagonal entries. Consequently, our framework naturally recovers SGD $(p=2)$ and signSGD $(p=\infty)$, while also encompassing a continuum of methods that interpolate between them.


    \item[{(iii)}] \textbf{Momentum and mini-batch gradient estimators.} The algorithms studied in this paper construct stochastic gradient estimates using momentum and mini-batching. At iteration $k$, a mini-batch gradient estimate $\bar G_k$
    is combined with the previous estimate through a momentum mechanism:
    \begin{align*}
    \bar G_k =  \frac{1}{B}\sum_{i=1}^B G(X_k;\xi_{k,i}),\quad M_k=(1-\theta)M_{k-1}+\theta \bar G_k\qquad\forall k\ge 1,
    \end{align*}
    where $G(\cdot;\cdot)$ is the stochastic gradient, $\theta\in(0,1]$, and $M_0$ is an initial estimate. Momentum and mini-batching are two standard tools for stabilizing stochastic gradient estimates in large-scale optimization. 
    

    \item[{(iv)}] \textbf{Multi-block extension.} The single‑block formulation isolates the role of matrix geometry, but a neural network typically consists of many parameter blocks with different dimensions, Lipschitz constants, and noise levels. We model this multi‑block setting as
    \begin{align*}
        \min_{\mathbf X\in
        \R^{m_1\times n_1}\times\cdots\times \R^{m_J\times n_J}}
        f(\mathbf X),
        \qquad
        \mathbf X=([X]_1,\ldots,[X]_J).
    \end{align*}
    To couple the blocks, we introduce weighted mixed norms that combine block-wise Schatten norms through a vector norm. Writing
    $L_j:=L_{p_j,q_j}$ for the block-wise Lipschitz constants, a primal-dual norm pair takes the form
    \begin{align*}
        \|\mathbf X\|_{{\bf w},{\bf p}}
        :=
        \left[
        \sum_{j=1}^J
        \left(L_j^{1/2}\|[X]_j\|_{S_{p_j}}\right)^\alpha
        \right]^{1/\alpha},\quad
        \|\mathbf G\|_{{\bf v},{\bf q}}
        :=
        \left[
        \sum_{j=1}^J
        \left(L_j^{-1/2}\|[G]_j\|_{S_{q_j}}\right)^\beta
        \right]^{1/\beta},
    \end{align*}
    where $\alpha$ and $\beta$, as well as $p_j$ and $q_j$, $j\in[J]$, are conjugate exponents. This construction is related in spirit to the modular-norm viewpoint of \cite{large2024scalable}, which argues that neural-network architecture should inform the scaling and normalization of updates. Here the weights encode heterogeneous block-wise smoothness and the outer norm specifies how the block-wise update magnitudes are coupled.
\end{enumerate}

The above ingredients reveal a gap in the existing literature. Prior work has clarified the role of norm geometry in stochastic optimization algorithms, motivated matrix-aware algorithms such as Muon, developed architecture-aware norms, and studied norm-constrained linear-minimization-oracle training \cite{bernstein2024old,jordan2024muon,large2024scalable,pethick2025training}. These works cover important aspects of the picture, but they do not provide a unified framework that combines Schatten-family oracle characterizations, momentum and mini-batching, matrix martingale control, and multi-block mixed-norm coupling.
\vspace{0.2cm}

\noindent \textbf{Contributions.} Our contributions mainly lie in the following aspects.
\begin{itemize}[leftmargin = .6cm]
    \item We derive closed‑form selections for two basic oracle primitives: the Schatten‑norm ball oracle and the Schatten‑norm regularization oracle (see Lemma \ref{lem:single-block-oracle}). We then extend these characterizations to the mixed-norm setting, obtaining closed‑form selections for the corresponding ball and regularization oracles for the multi‑block optimization problem (see Lemma \ref{lem:mult-block-subproblem}).
    \item We prove dimension-free moment bounds for matrix martingales under Schatten-$q$ norms, $q\in[1,2]$; see Lemma \ref{lem:concentration-Sp}. We also establish their mixed-norm counterparts; see Lemma~\ref{lem:multi-block-martingale}. These martingale moment bounds allow us to control the accumulation of stochastic matrix noise under Schatten and mixed-norm geometries, thereby enabling dimension-free convergence guarantees. The key ingredient is the technical inequality in Lemma~\ref{lem:nuc-ppt}, which serves as a substitute for the convenient square-expansion property available in the Euclidean/Frobenius norm setting.
    \item For single-block problems, we use these martingale bounds to derive dimension-free stationarity guarantees for methods based on the ball and regularization oracles; see Theorems \ref{thm:F-con} and \ref{thm:spec-con}. For the $S_p$-ball-oracle method, we prove
    \[
    \frac1K\sum_{k=0}^{K-1}
    \E[\|\nabla f(X_k)\|_{S_q}]
    \le
    \mathcal O\!\left(\sqrt{\sigma_{q,B}}\bigg(\frac{\Delta_f L_{p,q}}{K}\bigg)^{1/4}
    +\sqrt{\frac{\Delta_f L_{p,q}}{K}}
    +\frac{\sigma_{q,B}^2}{\sqrt{\Delta_f L_{p,q}K}}
    \right).
    \]
    For the $S_p$-regularization-oracle method, we prove
    \[
    \frac1K\sum_{k=0}^{K-1}
    \E[\|\nabla f(X_k)\|_{S_q}^2]
    \le
    \mathcal O\!\left(
    \sqrt{\frac{\sigma_{q,B}^2(L_{p,q}\Delta_f+\sigma_{q,B}^2)}{K}}
    +\frac{L_{p,q}\Delta_f+\sigma_{q,B}^2}{K}
    \right).
    \]
    Here, $\Delta_f := f(X_0) - f_{\mathrm{low}}$ denotes the initial function-value gap, while $\sigma_{q,B}^2$ and $L_{p,q}$ denote the variance proxy and the Lipschitz constant under the Schatten geometry, respectively. These bounds recover the classical dimension-free optimal $\mathcal{O}(K^{-1/4})$ convergence rate for the stationarity measure of stochastic first-order methods.
    \item We extend the oracle characterizations and martingale bounds to the weighted mixed-norm multi-block setting, and establish the corresponding stationarity guarantees; see Theorems \ref{thm:multi-F-con} and \ref{thm:multi-spec-con}. With
    $\sigma_{{\bf v},{\bf q},B}^2$ denoting the mixed-norm variance proxy, the ball-oracle-based method has the bound
    \begin{align*}
    \frac{1}{K}\sum_{k=0}^{K-1}
    \E[\|\nabla f({\rmX}_k)\|_{{\bf v},{\bf q}}]\le\mathcal O\!\left(\sqrt{\sigma_{{\bf v},{\bf q},B}}\bigg(\frac{\Delta_f}{K}\bigg)^{1/4}
    +\sqrt{\frac{\Delta_f}{K}}
    + \frac{\sigma_{{\bf v},{\bf q},B}^2}{\sqrt{\Delta_f K}}\right), 
    \end{align*}
    and the regularization-oracle-based method has the bound
    \begin{align*}
    \frac{1}{K}\sum_{k=0}^{K-1}
    \E[\|\nabla f({\rmX}_k)\|_{{\bf v},{\bf q}}^2] \le \mathcal O\!\left(
    \sqrt{\frac{\sigma_{{\bf v},{\bf q},B}^2(\Delta_f+\sigma_{{\bf v},{\bf q},B}^2)}{K}}
    + \frac{\Delta_f+\sigma_{{\bf v},{\bf q},B}^2}{K}\right).    
    \end{align*}
    These bounds recover the classical dimension-free optimal $\mathcal{O}(K^{-1/4})$ convergence rate for the stationarity measure of stochastic first-order methods.
\end{itemize}
\vspace{0.2cm}

\noindent \textbf{Organization.} The rest of this paper is organized as follows. Section \ref{sec:not-pre}
introduces notation and the normed-space preliminaries used throughout the
paper. In Section \ref{sec:oracle}, we derive the explicit expressions for the single-block and multi-block oracles. In Section \ref{sec:single}, we develop and analyze the Schatten-family $S_p$-$S_q$
methods with ball and regularization oracles for single-block matrix optimization problems. Section \ref{sec:multi} extends the methods and
analysis to solve multi-block matrix optimization problems.

\section{Notation and preliminaries}\label{sec:not-pre}

Throughout this paper, we use $\R^{m\times n}$ to represent the Euclidean space of $m\times n$ matrices. We use $\mathrm{Tr}(\cdot)$ to denote the trace of a square matrix, and $\langle\cdot,\cdot\rangle$ to denote the trace inner product for matrices. For any matrix $Z\in\mathbb{R}^{m\times n}$, we let $\sigma(Z)$ denote the vector of its singular values and we use $\|Z\|_{S_r} := \|\sigma(Z)\|_r$, where $\|\cdot\|_{r}$ is the vector $\ell_r$ norm, to denote its Schatten-$r$ norm (also written as the $S_r$-norm), for any $r\in[1,\infty]$.
The pair $p,q\in[1,\infty]$ is called conjugate if $1/p+1/q=1$, where $1/\infty:=0$. When $(p,q)$ is a conjugate pair, $S_p$- and $S_q$-norms form a dual pair in the sense
\begin{align}\label{ineq:dual-Sp-norm}
    \|G\|_{S_q} = \max_{\|X\|_{S_p}\le 1} \langle G,X\rangle\qquad\forall G\in\R^{m\times n}.
\end{align}
For any symmetric matrices $A_1,A_2\in\R^{m\times m}$, $A_1\preceq A_2$ means $A_2 - A_1$ is positive semidefinite. For $\alpha\ge0$ and any positive semidefinite (p.s.d.) matrix $A\in\R^{m\times m}$ with singular value decomposition $A=U\mathrm{diag}(\sigma_1,\ldots,\sigma_m)U^T$, we define the $\alpha$th power of $A$ by $A^\alpha = U\mathrm{diag}(\sigma_1^\alpha,\ldots,\sigma_m^\alpha)U^T$. 
For $\alpha<0$, we define $A^\alpha=(A^+)^{-\alpha}$, where $A^+$ denotes the Moore-Penrose pseudoinverse of $A$. 

We next introduce the norm structure for the Cartesian product space, denoted by
\begin{align*}
\MBspace := \R^{m_1\times n_1} \times \cdots \times \R^{m_J\times n_J}\quad \text{with}\quad {\bf m}:=(m_1,\ldots,m_J),\quad {\bf n}:=(n_1,\ldots,n_J),
\end{align*}
where $\{m_j\}_{j=1}^J$ and $\{n_j\}_{j=1}^J$ are sets of positive integers. For any ${\bf X}\in\MBspace$ and $j\in[J]$, we let $[X]_j\in\R^{m_j\times n_j}$ be its $j$th block. We define mixed norms on $\MBspace$ as follows
\begin{align*}
\|\mathbf{X}\|_{{\bf u}, {\bf r}} := \|(L_1^{1/2}\|[X]_1\|_{S_{r_1}},\ldots,L_J^{1/2}\|[X]_J\|_{S_{r_J}})\|_\gamma\qquad \forall\mathbf{X}\in\MBspace&\\
\text{with}\quad {\bf u}:=(\gamma,L_1,\ldots,L_J),\quad {\bf r}:=(r_1,\ldots,r_J),&
\end{align*}
where $\gamma\in[1,\infty]$, $\|\cdot\|_\gamma$ denotes the $\ell_\gamma$-norm on $\R^J$, $\{L_j\}_{j=1}^J$ is a set of positive numbers, $\{r_j\}_{j=1}^J$ is a set of extended real numbers in $[1,\infty]$. This mixed norm is formed by the composition of a weighted $\ell_\gamma$-norm and $S_{r_j}$-norms, $j\in[J]$, where the $S_{r_j}$-norm measures each $\R^{m_j\times n_j}$, and the weighted $\ell_\gamma$-norm controls the interaction across the $J$ block-wise subspaces. The inner product on $\MBspace$ is given as follows
\begin{align*}
\langle{\bf X}, {\bf Y}\rangle = \sum_{j=1}^J\langle [X]_j, [Y]_j\rangle \qquad\forall {\bf X},{\bf Y}\in\MBspace.
\end{align*}
When $(\alpha,\beta)$ and $\{(p_j,q_j)\}_{j=1}^J$ are conjugate pairs, $\|\cdot\|_{{\bf w}, {\bf p}}$ and $\|\cdot\|_{{\bf v}, {\bf q}}$ with
\begin{align*}
{\bf w}:= (\alpha,L_1,\ldots,L_J),\quad  {\bf v} := (\beta,L_1^{-1},\ldots,L_J^{-1}),\quad    {\bf p}:=(p_1,\ldots,p_J),\quad {\bf q}:=(q_1,\ldots,q_J)
\end{align*}
form a dual pair in the sense
\begin{align*}
    \|{\bf G}\|_{{\bf v}, {\bf q}} = \max_{\|{\bf X}\|_{{\bf w}, {\bf p}}\le 1} \langle {\bf G}, {\bf X}\rangle\qquad\forall {\bf G}\in\MBspace,
\end{align*}
which holds true by applying the trace Hölder-type inequalities twice
\begin{align*}
    \langle {\bf G}, {\bf X}\rangle = \sum_{j=1}^J\langle [G]_j, [X]_j\rangle \le \sum_{j=1}^J \|[G]_j\|_{S_{q_j}} \|[X]_j\|_{S_{p_j}} \le  \|{\bf G}\|_{{\bf v}, {\bf q}} \|{\bf X}\|_{{\bf w}, {\bf p}}\qquad\forall {\bf X}, {\bf G}\in\MBspace.
\end{align*}

\section{Oracles}\label{sec:oracle}

Deep learning optimization can be viewed through the lens of local search oracles, in which each iteration produces a parameter update that stays within a controlled distance of the current parameters. Because of the highly overparameterized and heterogeneous nature of modern neural networks, the choice of norm or distance metric used to guide parameter updates plays a central role in determining algorithmic performance.

In this section, we investigate the classical ball-oracle and regularization-oracle strategies for implementing local search, with distance-control norms chosen as Schatten-family norms for single-block spaces, and mixed norms for multi-block spaces as discussed in Section \ref{sec:not-pre}. We derive closed-form selections for the resulting oracles. 


\subsection{Single-block oracles}

This subsection concerns the single-block $S_p$-norm-based ball oracles and regularization oracles, which are defined as
\begin{align}\label{def:sol-ball-reg}
\Delta_p^{\rm b}(M) = \argmin_{\|Z\|_{S_p}\le 1}\ \langle M,Z\rangle,\quad \Delta_p^{\rm r}(M) = \argmin_{Z\in\R^{m\times n}} \bigg\{ \langle M, Z\rangle + \frac{1}{2}\|Z\|_{S_p}^2\bigg\}\qquad\forall M\in\R^{m\times n}.  
\end{align}

The next lemma provides the closed-form expressions for the above oracles.
\begin{lemma}\label{lem:single-block-oracle}
Let $M\in\mathbb{R}^{m\times n}$ be given, let $p\in[2,\infty]$ and
$q\in[1,2]$ be a conjugate pair, and let $\Delta_p^{\rm b}(\cdot)$ and
$\Delta_p^{\rm r}(\cdot)$ be defined in \eqref{def:sol-ball-reg}. If
$M=0$, we take the selections $\Delta_p^{\rm b}(0)=0$ and
$\Delta_p^{\rm r}(0)=0$. If $M\neq0$, one valid selection is given by
\begin{align}\label{def:Delta1M-Delta2M}
\Delta_p^{\rm b}(M) = - \frac{M(M^TM)^{\frac{q-2}{2}}}{\|M\|_{S_q}^{q-1}},\quad \Delta_p^{\rm r}(M) = \|M\|_{S_q}\cdot \Delta_p^{\rm b}(M),  
\end{align}
where any negative matrix power is interpreted through pseudo-inverse.
\end{lemma}

\subsection{Multi-block oracles}

This subsection concerns the multi-block mixed-norm-based ball oracle and regularization oracle, which are defined as
\begin{align}\label{def:sol-ball-reg-multi}
    \Delta_{{\bf w},{\bf p}}^{\rm b}({\bf M}) = \argmin_{\|{\bf Z}\|_{{\bf w},{\bf p}}\le 1}\, \langle {\bf M}, {\bf Z}\rangle,\quad \Delta_{{\bf w},{\bf p}}^{\rm r}({\bf M}) = \argmin_{{\bf Z}\in\MBspace} \bigg\{ \langle {\bf M}, {\bf Z}\rangle + \frac{1}{2}\|{\bf Z}\|_{{\bf w},{\bf p}}^2\bigg\}\quad\forall {\bf M}\in\MBspace,
\end{align}
where ${\bf p}=(p_1,\ldots,p_J)$ is a collection of extended real numbers in $[2,\infty]$, and ${\bf w}=(\alpha,L_1,\ldots,L_J)$ with $\alpha\in[2,\infty]$ and $\{L_j\}_{j=1}^J$ being a set of positive numbers. 

The following lemma gives the closed-form expressions for the above oracles.

\begin{lemma}\label{lem:mult-block-subproblem}
Let ${\bf M}\in\MBspace$ be given, and let
$\Delta_{{\bf w},{\bf p}}^{\rm b}(\cdot)$ and
$\Delta_{{\bf w},{\bf p}}^{\rm r}(\cdot)$ be defined in
\eqref{def:sol-ball-reg-multi}. Let $\|\cdot\|_{{\bf v}, {\bf q}}$ be
the dual norm of $\|\cdot\|_{{\bf w}, {\bf p}}$, where
${\bf q}=(q_1,\ldots,q_J)$ and
${\bf v}=(\beta,L_1^{-1},\ldots,L_J^{-1})$, with
$\{(p_j,q_j)\}_{j=1}^J$ and $(\alpha,\beta)$ being conjugate pairs. If
$\|{\bf M}\|_{{\bf v},{\bf q}}=0$, we take the selections
$\Delta_{{\bf w},{\bf p}}^{\rm b}({\bf M})=0$ and
$\Delta_{{\bf w},{\bf p}}^{\rm r}({\bf M})=0$. If
$\|{\bf M}\|_{{\bf v},{\bf q}}>0$, one valid selection of the ball oracle is
given by
\begin{align}\label{def:multi-ball-oracle}
\big[\Delta_{{\bf w},{\bf p}}^{\rm b}({\bf M})\big]_j
=
L_j^{-1/2}
\left(
\frac{L_j^{-1/2}\|[M]_j\|_{S_{q_j}}}
{\|{\bf M}\|_{{\bf v},{\bf q}}}
\right)^{\beta-1}
\Delta_{p_j}^{\rm b}([M]_j),
\qquad j\in[J],
\end{align}
Moreover, one valid selection of the regularization oracle is
\begin{align}\label{def:multi-reg-oracle}
\big[\Delta_{{\bf w},{\bf p}}^{\rm r}({\bf M})\big]_j
=
\|{\bf M}\|_{{\bf v},{\bf q}}\,
\big[\Delta_{{\bf w},{\bf p}}^{\rm b}({\bf M})\big]_j,
\qquad j\in[J].
\end{align}
\end{lemma}

\section{$S_p$-$S_q$ methods for single-block problems}\label{sec:single}

In this section, we propose $S_p$-$S_q$ methods\footnote{We call them the $S_p$-$S_q$ methods because their oracles can be formulated either as an $S_p$-norm ball-constrained subproblem or as an $S_p$-norm-squared regularized subproblem, while their closed-form solutions can be computed using the dual $S_q$-norm and the dual exponent $q$ (see Lemma \ref{lem:single-block-oracle}).} with momentum for solving the following single-block matrix optimization problem:
\begin{align}\label{pb:matrix-opt}
\min_{X\in\R^{m\times n}} f(X).    
\end{align}
Without loss of generality, we assume $m\le n$.

We make the following assumptions on problem \eqref{pb:matrix-opt} throughout this section.

\begin{assumption}\label{asp:basic}
\begin{enumerate}[leftmargin = .7cm]
    \item[{(a)}] There exists a finite $f_{\mathrm{low}}$ such that $f(X)\ge f_{\mathrm{low}}$ for all $X\in\mathbb{R}^{m\times n}$.
    \item[{(b)}] There exist $p\in[2,\infty]$, $q\in[1,2]$, and $L_{p,q}>0$ such that $p,q$ are a conjugate pair and $\|\nabla f(Y) - \nabla f(X)\|_{S_q} \le L_{p,q} \|Y - X\|_{{S_p}}$ for all $X,Y\in\mathbb{R}^{m\times n}$.
    \item[{(c)}] There exists  $V\in\R^{m\times m}$ such that the stochastic gradient estimator $G:\R^{m\times n}\times\Xi\to\R^{m\times n}$ satisfies the following covariance bound for all $X\in\mathbb{R}^{m\times n}$:
    \begin{align*}
        \E[G(X;\xi)] = \nabla f(X),\quad \E[(G(X;\xi)-\nabla f(X))(G(X;\xi)-\nabla f(X))^T]\preceq VV^T.
    \end{align*} 
\end{enumerate}
\end{assumption}

\begin{remark}
(i) Assumption \ref{asp:basic}(b) states the Lipschitz continuity of $\nabla f$ with respect to the $S_p$- and $S_q$-norms. This $S_p$-$S_q$ Lipschitz assumption is consistent with the $S_p$-$S_q$ methods developed in this section, whose design and analysis employ the $S_p$-norm in the primal space (i.e., the $X$-variable space) and the $S_q$-norm in the dual space (i.e., the gradient space). It generalizes the Lipschitz assumptions adopted for Muon and SGD (e.g., see \cite{bernstein2024old}), and allows us to unify their study. As a result of Assumption \ref{asp:basic}(b), the quadratic upper bound holds:
\begin{align}\label{ineq:desc-lpq}
f(Y) \le f(X) + \langle\nabla f(X), Y-X\rangle + \frac{L_{p,q}}{2}\|Y-X\|_{S_p}^2,
\end{align}
due to the duality relationship between the $S_p$ and $S_q$-norms; see \eqref{ineq:dual-Sp-norm}.\vspace{0.1cm}

\noindent(ii) Assumption \ref{asp:basic}(c) means that the stochastic gradient $G$ is unbiased and its covariance is bounded above by $VV^T$ in the semidefinite partial order. The latter condition has previously been imposed in fine-grained analyses of matrix-variate methods; see, e.g., \cite{an2025asgo}. This matrix-valued covariance bound preserves the directional and spectral information of the stochastic noise, thereby facilitating a unified analysis framework for the $S_p$-$S_q$ methods developed in this section.

Moreover, the matrix-valued covariance bound assumption is not restrictive. Indeed, it is closely related to the bounded-variance condition $\E[\|G(X;\xi)-\nabla f(X)\|_F^2]\le\sigma_F^2$ for some $\sigma_F>0$. On the one hand, if $\E[(G(X;\xi)-\nabla f(X))(G(X;\xi)-\nabla f(X))^T]\preceq VV^T$, then taking the trace on both sides yields $\E[\|G(X;\xi)-\nabla f(X)\|_F^2]\le\|V\|_F^2$. On the other hand, $\E[\|G(X;\xi)-\nabla f(X)\|_F^2]\le\sigma^2_F$ 
implies $\E[(G(X;\xi)-\nabla f(X))(G(X;\xi)-\nabla f(X))^T]\preceq \sigma_F^2 I$.


\end{remark}

Our $S_p$-$S_q$ methods use the momentum scheme for constructing directions, which generates
\begin{align*}
M_k = (1-\theta) M_{k-1} + \frac{\theta}{B}\sum_{i=1}^B G(X_k;\xi_{k,i})\qquad\forall k\ge1,
\end{align*}
where $\{\xi_{k,i}\}_{i=1}^B$ is the set of independent samples drawn at iteration $k$. For convenience in the subsequent analysis, we adopt the following shorthand
\begin{align}
&\bar G_k := \frac{1}{B}\sum_{i=1}^B G(X_k;\xi_{k,i}),\quad \Delta M_k := M_k - \nabla f(X_k),\quad \Delta G_k := \bar G_k - \nabla f(X_k)\qquad\forall k\ge0.\label{def:Delta-not}
\end{align}
Let $\{\mathcal F_k\}_{k\ge-1}$ be defined by
$\mathcal F_{-1}:=\sigma(X_0)$ and
$\mathcal F_k:=\sigma(X_0,\{\xi_{s,i}:0\le s\le k,\ 1\le i\le B\})$  for
$k\ge0$. It then follows from Assumption \ref{asp:basic}(c) that
\begin{align*}
    \E[\Delta G_k\mid\mathcal F_{k-1}]=0,\quad \E[\Delta G_k\Delta G_k^T\mid\mathcal F_{k-1}]\preceq VV^T/B\qquad\forall k\ge0.
\end{align*}

\subsection{Martingale moment inequality}\label{subsec:mmi-single}

In this subsection, we establish a matrix martingale moment inequality that will be crucial in controlling the noise bound in our later algorithmic analysis.

The following lemma, which generalizes \cite[Lemma 8]{an2025asgo}, provides a useful property of Schatten norms. Its proof is deferred to Section \ref{sec:pf-single-mmi}.

\begin{lemma}\label{lem:nuc-ppt}
    Let $A\in\mathbb{R}^{m\times n}$ and $q \in [1, 2]$ be given. Then, for any symmetric $\Lambda \in\mathbb{R}^{m\times m}$ satisfying $\Lambda \succ 0$, it holds that
    \begin{align*}
    \|A\|_{S_q}^2 \le \|\Lambda\|_{S_{\frac{q}{2-q}}}\mathrm{Tr}(AA^T\Lambda^{-1}).
    \end{align*}
    Here $S_{\frac{q}{2-q}}$ is interpreted as $S_\infty$ when $q=2$.
\end{lemma}

The following lemma establishes a matrix martingale moment inequality with respect to the Schatten-$q$ norms for $q \in [1, 2]$. Its proof can be found in Section \ref{sec:pf-single-mmi}.

\begin{lemma}\label{lem:concentration-Sp}
Consider a sequence $\{A_t\}$ of random $m \times n$ matrices and let $\{\mathcal{F}_t\}$ denote its natural filtration, i.e., $\mathcal{F}_t := \sigma(A_0,\ldots,A_t)$. Assume $\{A_t\}$ is a martingale difference sequence, i.e., $\mathbb{E}[A_0] = 0$ and, for each $t\ge1$, $\E[A_t\ |\mathcal{F}_{t-1}]=0$. Furthermore, assume that $\mathbb{E}[A_t A_t^T]$ is well-defined for every $t$. Then, for any symmetric $\Lambda \succ 0$ and $q \in [1,2]$, it holds that
\begin{align}
 \E\Big[\Big\|\sum_{t=0}^{k-1}A_t\Big\|_{S_q}^2\Big] \le \|\Lambda\|_{S_{\frac{q}{2-q}}}\sum_{t=0}^{k-1}\mathrm{Tr}(\E[A_tA_t^T]\Lambda^{-1})\qquad\forall \;k\ge1. \label{ineq:concentration-Sp}
\end{align}
\end{lemma}

\subsection{An $S_p$-$S_q$ method with ball oracle}\label{sec:lmo}

In this subsection, we propose an $S_p$-$S_q$ method with ball oracle and establish its convergence guarantee. At each iteration $k$, this method first constructs a mini-batch stochastic gradient $\bar G_k$ based on a set of stochastic gradients $\{G(X_k;\xi_{k,i})\}$, and then performs a momentum update to obtain $M_k$ as a weighted average of past stochastic gradients. The algorithm then updates the next iterate $X_{k+1}$ using an $S_p$-norm-based ball oracle. Details of this method are presented in Algorithm \ref{alg:sdm-linear}.

\begin{algorithm}[!htbp] 
\caption{An $S_p$-$S_q$ method with ball oracle} 
\label{alg:sdm-linear} 
\begin{algorithmic}[0] 
\State \textbf{Input:} starting iterate $X_0\in\mathbb{R}^{m \times n}$, momentum parameter $\theta\in(0,1]$, Schatten norm parameter $p \in [2, \infty]$, step size $\eta>0$, batch size $B\ge1$. 
\For{$k=0,1,2,\ldots$} 
\State Sample a mini-batch $\{\xi_{k,i}\}_{i=1}^B$ and compute $\bar G_k=\frac{1}{B}\sum_{i=1}^B G(X_k;\xi_{k,i})$.
\State Compute the search direction: 
\begin{align}\label{update-mk-1} 
M_k = \begin{cases}
(1 - \theta) M_{k-1} + \theta \bar G_k, &\text{if } k \geq 1\\
\bar G_0 & \text{if } k =0.
\end{cases}
\end{align}
\State Update the next iterate:
\begin{align}\label{eq:update-xk-1}
X_{k+1}  = \argmin_{\|X-X_k\|_{S_p}\le\eta} \langle M_k, X-X_k\rangle.
\end{align} 
\EndFor
\end{algorithmic} 
\end{algorithm}  

The following lemma provides a descent inequality for the function value when updating with the $S_p$-norm-based ball oracle of Algorithm \ref{alg:sdm-linear}. Its proof is relegated to Section \ref{sec:pf-single-lmo}.

\begin{lemma}\label{lem:desc-ineq-Sp-ball}
Suppose that Assumption \ref{asp:basic} holds. Let $\{(X_k,M_k)\}$ be generated by Algorithm \ref{alg:sdm-linear} with step size $\eta>0$, and let $\{\Delta M_k\}$ be defined in \eqref{def:Delta-not} and $L_{p,q}$ be given in Assumption \ref{asp:basic}. Then, 
\begin{align}\label{ineq:matrix-sign-descent}
f(X_{k+1}) \le f(X_k) - \eta \|\nabla f(X_k)\|_{S_q} + 2\eta \|\Delta M_k\|_{S_q} + \frac{L_{p,q}\eta^2}{2}\qquad\forall k\ge0. 
\end{align}
\end{lemma}

The next lemma gives an upper bound on the estimation error $\{\Delta M_k\}$ for the momentum directions generated by Algorithm \ref{alg:sdm-linear}. Its proof is relegated to Section \ref{sec:pf-single-lmo}.

\begin{lemma}\label{lem:momeutm-bd-Sp-ball}
Suppose that Assumption \ref{asp:basic} holds. Let $\{(X_k,M_k)\}$ be generated by Algorithm \ref{alg:sdm-linear} with input parameters $\eta>0$ and $\theta\in(0,1]$, and let $\{(\Delta M_k,\Delta G_k)\}$ be defined in \eqref{def:Delta-not} and $L_{p,q}$ be given in Assumption \ref{asp:basic}. Then, 
\begin{align}\label{ineq:akeyrelation}
\|\Delta M_k\|_{S_q} \le (1-\theta)^k\|\Delta M_0\|_{S_q} + \theta \Big\|\sum_{t=0}^{k-1}(1-\theta)^t\Delta G_{k-t}\Big\|_{S_q} + L_{p,q}\eta \sum_{t=0}^{k-1} (1-\theta)^{t+1}.   
\end{align}
\end{lemma}

The next lemma gives an upper bound on the sum of the expected estimation error over the momentum directions generated by Algorithm \ref{alg:sdm-linear}. Its proof is relegated to Section \ref{sec:pf-single-lmo}.

\begin{lemma}\label{lem:F-con}
Suppose that Assumption \ref{asp:basic} holds. Let $\{(X_k,M_k)\}$ be generated by Algorithm \ref{alg:sdm-linear} with input parameters $\eta>0$, $\theta\in(0,1]$, and $B\ge1$. Let $\{\Delta M_k\}$ be defined in \eqref{def:Delta-not}, and let $L_{p,q}$ and $V$ be given in Assumption \ref{asp:basic}. Then,   
\begin{align}\label{ineq:expec-upbd-nuclear-norm}
\sum_{k=0}^{K-1}\E[\|\Delta M_k\|_{S_q}] \le \frac{\sigma_{q,B}}{\theta} + \sqrt{\theta} K \sigma_{q,B} + \frac{L_{p,q}\eta K}{\theta} \qquad\forall K\ge1,
\end{align}
where $\sigma_{q,B} := \|V\|_{S_q}/\sqrt{B}$.
\end{lemma}

The following theorem provides an upper bound on the expected $S_q$-norm-based stationarity achieved by Algorithm \ref{alg:sdm-linear}. Its proof is deferred to Section \ref{sec:pf-single-lmo}.

\begin{theorem}\label{thm:F-con}
Suppose that Assumption \ref{asp:basic} holds. Let $K$ be the maximum iteration number for running Algorithm \ref{alg:sdm-linear} and let $L_{p,q}$ and $V$ be as given in Assumption \ref{asp:basic}. Let $\{X_k\}$ be the iterates generated by Algorithm \ref{alg:sdm-linear} with input parameters $\eta$ and $\theta$ given by
\begin{align}\label{def:eta-theta-F}
\eta=\sqrt{\frac{\Delta_f\theta}{L_{p,q}K}} \quad \text{and} \quad \theta = \min\bigg\{\frac{1}{\sigma_{q,B}}\sqrt{\frac{\Delta_fL_{p,q}}{K}}, 1\bigg\},
\end{align}
where $\Delta_f := f(X_0) - f_{\mathrm{low}}$ and $\sigma_{q,B} := \|V\|_{S_q}/\sqrt{B}$.
Then, it holds that  for all $K\ge1$,
\begin{align}\label{ineq:ave-stat-F}
\frac{1}{K}\sum_{k=0}^{K-1}\E[\|\nabla f(X_k)\|_{S_q}] & \le 6\sqrt{\sigma_{q,B}}\Big(\frac{\Delta_fL_{p,q}}{K}\Big)^{1/4} + 8\sqrt{\frac{\Delta_fL_{p,q}}{K}} + \frac{2\sigma_{q,B}^2}{\sqrt{\Delta_fL_{p,q}K}}.
\end{align}    
\end{theorem}



\subsection{An $S_p$-$S_q$ method with regularization oracle}\label{sec:qmo}

In this subsection, we propose an $S_p$-$S_q$ method with regularization oracle and establish its convergence guarantees. At each iteration $k$, this method first constructs a mini-batch stochastic gradient $\bar G_k$ based on a set of stochastic gradients $\{G(X_k;\xi_{k,i})\}$, and then performs a momentum update to obtain $M_k$ as a weighted average of past stochastic gradients. The algorithm then updates the next iterate $X_{k+1}$ by an $S_p$-norm-based regularization oracle. Details of this method are presented in Algorithm \ref{alg:sdm-quadratic}.



\begin{algorithm}[!htbp] 
\caption{An $S_p$-$S_q$ method with regularization oracle} 
\label{alg:sdm-quadratic} 
\begin{algorithmic}[0] 
\State \textbf{Input:} starting iterate $X_0\in\mathbb{R}^{m \times n}$, momentum parameter $\theta\in(0,1]$, Schatten norm parameter $p \in [2, \infty]$, step size $\eta>0$, batch size $B\ge1$. 
\For{$k=0,1,2,\ldots$} 
\State Sample a mini-batch $\{\xi_{k,i}\}_{i=1}^B$ and compute $\bar G_k=\frac{1}{B}\sum_{i=1}^B G(X_k;\xi_{k,i})$.
\State Compute the search direction: 
\begin{align}
\label{update-mk-2} 
M_k = \begin{cases}
(1 - \theta) M_{k-1} + \theta \bar G_k, &\text{if } k \geq 1\\
\bar G_0 & \text{if } k =0.
\end{cases}
\end{align}
\State Update the next iterate:
\begin{align}\label{update-xk-2}
X_{k+1}  = \argmin_{X\in\R^{m\times n}} \bigg\{ \langle M_k, X-X_k\rangle + \frac{1}{2\eta}\|X-X_k\|_{S_p}^2\bigg\}.
\end{align} 
\EndFor
\end{algorithmic} 
\end{algorithm}

The following lemma provides a descent inequality for the function value when updating with the $S_p$-norm-based regularization oracle of Algorithm \ref{alg:sdm-quadratic}. Its proof is deferred to Section \ref{sec:pf-single-qmo}.

\begin{lemma}\label{lem:ineq-descent-qmo-single}
Suppose that Assumption \ref{asp:basic} holds. Let $\{(X_k,M_k)\}$ be generated by Algorithm \ref{alg:sdm-quadratic} with step size $\eta\in(0,1/L_{p,q}]$, and let $\{\Delta M_k\}$ be defined in \eqref{def:Delta-not} and $L_{p,q}$ be given in Assumption \ref{asp:basic}. Then, for any $k \geq0$, it holds that
\begin{align}
    f(X_{k+1}) 
    &\le f(X_k)
    - \Big(\frac{1}{4\eta} - \frac{L_{p,q}}{4}\Big)
    \|X_{k+1} - X_k\|_{S_p}^2 \nonumber \\
    &\qquad 
    - \frac{\eta}{8}(1 - L_{p,q}\eta)
    \|\nabla f(X_k)\|_{S_q}^2 
    + \frac{\eta}{4}(3 - L_{p,q}\eta)
    \|\Delta M_k\|_{S_q}^2.\label{ineq:matrix-sign-descent-qo}
\end{align}
\end{lemma}

The next lemma provides an upper bound on the squared estimation error $\{\Delta M_k\}$ for the momentum directions generated by Algorithm \ref{alg:sdm-quadratic}. Its proof is deferred to Section \ref{sec:pf-single-qmo}.

\begin{lemma}\label{lem:momeutm-bd-Sp-reg}
Suppose that Assumption \ref{asp:basic} holds. Let $\{(X_k,M_k)\}$ be generated by Algorithm \ref{alg:sdm-quadratic} with input parameters $\eta>0$ and $\theta\in(0,1]$, and let $\{(\Delta M_k,\Delta G_k)\}$ be defined in \eqref{def:Delta-not} and $L_{p,q}$ be given in Assumption \ref{asp:basic}. Then, 
\begin{align}\label{ineq:square-Sp-deter}
    \|\Delta M_k\|_{S_q}^2
    &\le 3(1-\theta)^{2k}\|\Delta M_0\|_{S_q}^2
    + 3\theta^2
    \Big\|\sum_{t=0}^{k-1}(1-\theta)^t\Delta G_{k-t}\Big\|_{S_q}^2 \nonumber\\
    &\qquad
    + \frac{3(1-\theta)^2 L_{p,q}^2}{\theta}
    \sum_{t=0}^{k-1}(1-\theta)^t
    \|X_{k-t-1} - X_{k-t}\|_{S_p}^2
\end{align}
    
\end{lemma}

The next lemma gives an upper bound on the sum of the expected squared estimation error over the momentum directions generated by Algorithm \ref{alg:sdm-quadratic}. Its proof is deferred to Section \ref{sec:pf-single-qmo}.

\begin{lemma}\label{lem:F-con-reg}
Suppose that Assumption \ref{asp:basic} holds. Let $\{(X_k,M_k)\}$ be generated by Algorithm \ref{alg:sdm-quadratic} with input parameters $\eta>0$, $\theta\in(0,1]$, and $B\ge1$. Let $\{\Delta M_k\}$ be defined in \eqref{def:Delta-not}, and let $L_{p,q}$ and $V$ be given in Assumption \ref{asp:basic}. Then,   
\begin{align}\label{ineq:expec-upbd-spec-norm}
\sum_{k=0}^{K-1} \E[\|\Delta M_k\|_{S_q}^2]
&\le \frac{3}{\theta}\sigma_{q,B}^2
+ 3\theta K \sigma_{q,B}^2
+ \frac{3 L_{p,q}^2}{\theta^2}
\sum_{k=1}^{K-1}
\E[\|X_k - X_{k-1}\|_{S_p}^2]
\qquad\forall K\ge1,    
\end{align}
where $\sigma_{q,B}:=\|V\|_{S_q}/\sqrt{B}$.
\end{lemma}

The next theorem provides an upper bound on the expected squared $S_q$-norm-based stationarity achieved by Algorithm \ref{alg:sdm-quadratic}. Its proof is deferred to Section \ref{sec:pf-single-qmo}.

\begin{theorem}\label{thm:spec-con}
Suppose that Assumption \ref{asp:basic} holds. Let $K$ be the maximum iteration number for running Algorithm \ref{alg:sdm-quadratic} and let $L_{p,q}$ and $V$ be given in Assumption \ref{asp:basic}. Let $\{X_k\}$ be the iterates generated by Algorithm \ref{alg:sdm-quadratic} with input parameters $\eta$ and $\theta$ given by
\begin{align}\label{def:eta-theta-spec}
\eta=\frac{\theta}{4L_{p,q}}\quad \text{and}\quad
\theta = \min\left\{
\sqrt{\frac{L_{p,q}\Delta_f+\sigma_{q,B}^2}
{\sigma_{q,B}^2K}},1
\right\},
\end{align}
where $\Delta_f:=f(X_0)-f_{\mathrm{low}}$ and $\sigma_{q,B}:=\|V\|_{S_q}/\sqrt{B}$.
Then, 
\begin{align}\label{ineq:ave-stat-spec}
\frac{1}{K}\sum_{k=0}^{K-1}
\E[\|\nabla f(X_k)\|_{S_q}^2]
& \le 67\left[
\sqrt{\frac{\sigma_{q,B}^2
(L_{p,q}\Delta_f+\sigma_{q,B}^2)}{K}}
+ \frac{L_{p,q}\Delta_f+\sigma_{q,B}^2}{K}
\right]\qquad\forall K\ge1.
\end{align}    
\end{theorem}

\section{Multi-block coupling via weighted norms}\label{sec:multi}

In this section, we consider the multi-block matrix optimization problem
\begin{align}\label{pb:matrix-opt-multi}
\min_{\mathbf{X}\in\MBspace} f(\mathbf{X}),    
\end{align}
which is a generalization of the single-block problem \eqref{pb:matrix-opt} that more closely aligns with the structure of neural networks. Without loss of generality, we assume $m_j\le n_j$ for each $j$. For each block matrix space $\R^{m_j\times n_j}$, adopt the $S_{p_j}$- and $S_{q_j}$-norms as metrics, with the intention of extending the $S_{p_j}$-$S_{q_j}$ methods developed in Section \ref{sec:single} for optimizing the component $[X]_j$. To connect the $J$ blocks, we couple the $S_{p_j}$- and $S_{q_j}$-norms, $j\in[J]$, with weighted $\ell_\alpha$- and $\ell_\beta$-norms, and formulate a mixed norm
\begin{align}\label{def:mixed-norm}
\|{\bf X}\|_{{\bf w},{\bf p}} := \Big[\sum_{j=1}^J (L_j^{1/2}\|[X]_j\|_{S_{p_j}})^\alpha\Big]^{\frac{1}{\alpha}},\quad \|{\bf X}\|_{{\bf v}, {\bf q}} := \Big[\sum_{j=1}^J (L_j^{-1/2}\|[X]_j\|_{S_{q_j}})^\beta\Big]^{\frac{1}{\beta}},
\end{align}
where
\begin{align*}
&{\bf p}:=(p_1,\ldots,p_J),\quad {\bf q}:=(q_1,\ldots,q_J),\quad L_j:=L_{p_j,q_j},\\
&{\bf w}:=(\alpha,L_1,\ldots,L_J),\quad {\bf v}:=(\beta,L_1^{-1},\ldots,L_J^{-1}).
\end{align*}
The weights $\{L_{p_j,q_j}\}_{j=1}^J$ allow heterogeneous Lipschitz moduli associated with the $J$ block-wise subspaces to be incorporated into the mixed norms, while the $\ell_\alpha$- and $\ell_\beta$-norms naturally impose different coupling rules. The mixed norms will be very useful in designing algorithmic strategies for solving \eqref{pb:matrix-opt-multi}. Special cases of the coupling rules, obtained by specifying $(\alpha,\beta)$, will be further illustrated in Sections \ref{sec:multi-ball} and \ref{sec:multi-reg}, respectively.

We make the following assumption on problem \eqref{pb:matrix-opt-multi} throughout this paper.

\begin{assumption}\label{asp:multi-basic}
\begin{enumerate}[leftmargin = .7cm]
    \item[{(a)}] There exists a finite $f_{\mathrm{low}}$ such that $f(\mathbf X)\ge f_{\mathrm{low}}$ for all $\mathbf X\in\MBspace$.
    \item[{(b)}] The parameters $\alpha\in[2,\infty]$ and $\beta\in[1,2]$ are a conjugate pair, and $p_j\in[2,\infty]$ and $q_j\in[1,2]$ are also a conjugate pair for every $j\in[J]$.
    \item[{(c)}] The gradient $\nabla f$ satisfies the Lipschitz condition:
    \[
    \|\nabla f(\mathbf Y)-\nabla f(\mathbf X)\|_{{\bf v},{\bf q}}
    \le
    \|\mathbf Y-\mathbf X\|_{{\bf w},{\bf p}}
    \qquad
    \forall \mathbf X,\mathbf Y\in\MBspace.
    \]
    \item[{(d)}] There exist some $V_j\in\R^{m_j\times m_j}$, $j\in[J]$ such that the stochastic gradient estimator $\mathbf G:\MBspace\times\Xi\to\MBspace$ satisfies the following covariance bound for all $\mathbf{X}\in\MBspace$:
    \[
    \E[\mathbf G(\mathbf X;\xi)]=\nabla f(\mathbf X),
    \quad
    \E\!\left[
    ([\mathbf G(\mathbf X;\xi)]_j-[\nabla f(\mathbf X)]_j)
    ([\mathbf G(\mathbf X;\xi)]_j-[\nabla f(\mathbf X)]_j)^T
    \right]\preceq V_j V_j^T.
    \]
\end{enumerate}
\end{assumption}

\begin{remark}(i) Assumption \ref{asp:multi-basic}(c) states a Lipschitz-type condition on $\nabla f$, naturally induced by the dual norm pair $\|\cdot\|_{{\bf w},{\bf p}}$ and $\|\cdot\|_{{\bf v},{\bf q}}$. The Lipschitz constants of $\nabla f$ for the $J$ matrix blocks are encoded in the norms as weights in the $\ell_\alpha$- and $\ell_\beta$-norms; see \eqref{def:mixed-norm}. As a consequence of Assumption \ref{asp:multi-basic}(c), the quadratic upper bound holds:
    \begin{align}\label{ineq:desc-multi}
    f(\mathbf Y)\le f(\mathbf X) + \langle\nabla f(\mathbf X),\mathbf Y-\mathbf X\rangle + \frac{1}{2}\|\mathbf Y-\mathbf X\|_{{\bf w},{\bf p}}^2\qquad \forall \mathbf X,\mathbf Y\in\MBspace.
    \end{align}

    \noindent(ii) Assumption \ref{asp:multi-basic}(d) is a direct extension of Assumption \ref{asp:basic}(c) to the multi-block setting.
\end{remark}

Our multi-block methods use a momentum scheme for constructing directions, which generates a sequence:
\begin{align*}
{\bf M}_k = (1-\theta) {\bf M}_{k-1} + \frac{\theta}{B}\sum_{i=1}^B {\bf G}({\bf X}_k;\xi_{k,i})\qquad\forall k\ge1,
\end{align*}
where $\{\xi_{k,i}\}_{i=1}^B$ is the set of independent samples drawn at iteration $k$. For convenience in the later analysis, we adopt the following shorthand
\begin{align}
&\bar{\mathbf G}_k
:= \frac{1}{B}\sum_{i=1}^B \mathbf G(\mathbf X_k;\xi_{k,i}),
\quad
\Delta \mathbf M_k
:= \mathbf M_k-\nabla f(\mathbf X_k),
\quad
\Delta \mathbf G_k
:= \bar{\mathbf G}_k-\nabla f(\mathbf X_k).
\label{def:Delta-not-multi}
\end{align}
Let $\{\mathcal F_k\}_{k\ge-1}$ be defined by $\mathcal F_{-1}:=\sigma(\mathbf X_0)$ and
$\mathcal F_k:=\sigma(\mathbf X_0,\{\xi_{s,i}:0\le s\le k,\ 1\le i\le B\})$
for $k\ge0$. Then, as a result of Assumption \ref{asp:multi-basic}(d), we have
\[
\E[\Delta \mathbf G_k\mid\mathcal F_{k-1}]=0,\quad \E[[\Delta \mathbf G_k]_j[\Delta \mathbf G_k]_j^T\mid\mathcal F_{k-1}]
\preceq V_j V_j^T/B
\qquad \forall k \geq 0,\; j\in[J].
\]

\subsection{Martingale moment inequality}\label{sec:multi-mii}

In this subsection, we establish a martingale moment inequality for random sequences in $\R^{{\bf m}\odot{\bf n}}$, which will be crucial for controlling the noise term in our subsequent algorithmic analysis. The proof can be found in Section \ref{sec:pf-multi-mii}.

\begin{lemma}\label{lem:multi-block-martingale}
Consider a sequence of random multi-block matrices $\{\mathbf A_t\}\subseteq$\MBspace.
Write
\[
\mathbf A_t=([A_t]_1,\ldots,[A_t]_J),
\qquad
[A_t]_j\in\R^{m_j\times n_j}.
\]
Let
$\{\mathcal F_t\}$ denote its natural filtration, i.e.,
$\mathcal F_t:=\sigma(\mathbf A_0,\ldots,\mathbf A_t)$.
Assume $\{\mathbf A_t\}$ is a block-wise martingale difference sequence:
$\E[[A_0]_j]=0$ and, for each $t\ge1$,
$\E[[A_t]_j\mid\mathcal F_{t-1}]=0$ for all $j\in[J]$.
In addition, assume that $\E[[A_t]_j[A_t]_j^T]$ is well-defined
for any $t$ and $j$. Then, for arbitrary $k\geq 1$ and symmetric matrices
$\Lambda_j\succ0$, $q_j\in[1,2]$, define
\[
B_j^2
:=
L_j^{-1}
\|\Lambda_j\|_{S_{\frac{q_j}{2-q_j}}}
\sum_{t=0}^{k-1}
\operatorname{Tr}\!\left(
\E[[A_t]_j[A_t]_j^T]\Lambda_j^{-1}
\right),
\]
where $S_{\frac{q_j}{2-q_j}}$ is interpreted as $S_\infty$ when
$q_j=2$. Then, for any $\beta\in[1,2]$, it holds that 
\[
\E\Big[
\Big\|
\sum_{t=0}^{k-1}\mathbf A_t
\Big\|_{{\bf v},{\bf q}}^2
\Big]
\le
\|(B_1,\ldots,B_J)\|_\beta^2.
\]
\end{lemma}

\subsection{Multi-block coupling with ball oracle}\label{sec:multi-ball}

In this subsection, we propose and analyze a multi-block method with ball oracle. This method is a multi-block extension of the $S_p$-$S_q$ method with ball oracle (Algorithm \ref{alg:sdm-linear}) for solving problem \eqref{pb:matrix-opt-multi}, where the $S_p$-norm used to construct the ball oracle is replaced by the mixed norm $\|\cdot\|_{{\bf w},{\bf p}}$. Details of this method are presented in Algorithm \ref{alg:multi-sdm-linear}.


We now make a few remarks on the solution of the ball-constrained subproblem \eqref{eq:update-xk-multi-1}. In fact, the solution can always be obtained by rescaling the solutions of $J$ separate $S_{p_j}$-norm ball oracle problems, one for each block. The weighted $\ell_\alpha$-norm structure in $\|\cdot\|_{{\bf w},{\bf p}}$ only affects the step size for each block. For illustration, we consider the cases of the weighted $\ell_\alpha$-norm separately for $\alpha=\infty$ and $\alpha\in[2,\infty)$:
\begin{itemize}[leftmargin = .6cm]
    \item $\ell_\infty$-norm: In this case, \eqref{eq:update-xk-multi-1} employs the mixed norm $\|\cdot\|_{{\bf w},{\bf p}}$ with ${\bf w}=(\infty,L_1,\ldots,L_J)$. By the definition of $\|\cdot\|_{{\bf w},{\bf p}}$, \eqref{eq:update-xk-multi-1} is immediately separable. Consequently, we obtain
    \begin{align*}
    [X_{k+1}]_j = \argmin_{\|[X]_j- [X_k]_j\|_{S_{p_j}}\le L_j^{-1/2}\eta}
    \langle [M_k]_j, [X]_j- [X_k]_j\rangle\qquad  \forall j\in[J].
    \end{align*}
    Therefore, the $\ell_\infty$-norm coupling simply yields a step size of $L_j^{-1/2}\eta$ for the $j$th block, where $L_j$ is the block-wise Lipschitz constant associated with the $j$th block.
    
    \item $\ell_\alpha$-norm ($\alpha\in[2,\infty)$): In this case, although the resulting subproblem appears non-separable, Lemma \ref{lem:mult-block-subproblem} shows that its solution can nevertheless be obtained by solving $J$ block-wise subproblems. Specifically, solving \eqref{eq:update-xk-multi-1} is equivalent to solving the following $J$ subproblems:
    \begin{align*}
        [X_{k+1}]_j = \argmin_{\|[X]_j- [X_k]_j\|_{S_{p_j}}\le L_j^{-1/2}\omega_j\eta}
    \langle [M_k]_j, [X]_j- [X_k]_j\rangle\qquad  \forall j\in[J],
    \end{align*}
    where $\omega_j$'s are rescaling factors given by
    \begin{align*}
        \omega_j := \left( \frac{L_j^{-1/2}\|[M]_j\|_{S_{q_j}}}
{\|{\bf M}\|_{{\bf v},{\bf q}}}\right)^{\beta-1}\qquad \forall j\in[J],
\end{align*}
    and the vector $(\omega_1,\ldots,\omega_J)$ satisfies $\|(\omega_1,\ldots,\omega_J)\|_\alpha=1$. Therefore, $\ell_\alpha$-norm-based coupling yields a step size of $L_j^{-1/2}\omega_j\eta$, where the rescaling factors $\{\omega_j\}_{j=1}^J$ form a unit vector in terms of $\ell_\alpha$-norm, and $\omega_j$ is larger when $\|[M]_j\|_{S_{q_j}}$ is larger.
\end{itemize}   



\begin{algorithm}[!htbp] 
\caption{A multi-block method with ball oracle} 
\label{alg:multi-sdm-linear} 
\begin{algorithmic}[0] 
\State \textbf{Input:} starting iterate $\mathbf X_0\in\MBspace$, momentum parameter $\theta\in(0,1]$, norm parameters $\alpha\in[2,\infty]$ and $\mathbf p\in[2,\infty]^J$, step size $\eta>0$, batch size $B\ge1$. 
\For{$k=0,1,2,\ldots$} 
\State Sample a mini-batch $\{\xi_{k,i}\}_{i=1}^B$ and compute $\bar{\mathbf G}_k=\frac{1}{B}\sum_{i=1}^B \mathbf G(\mathbf X_k;\xi_{k,i})$.
\State Compute the search direction: 
\begin{align}\label{update-mk-multi-1} 
\mathbf M_k = \begin{cases}
(1 - \theta) \mathbf M_{k-1} + \theta \bar{\mathbf G}_k, &\text{if } k \geq 1,\\
\bar{\mathbf G}_0, & \text{if } k =0.
\end{cases}
\end{align}
\State Update the next iterate:
\begin{align}\label{eq:update-xk-multi-1}
\mathbf X_{k+1}
= \argmin_{\|\mathbf X-\mathbf X_k\|_{{\bf w},{\bf p}}\le\eta}
\langle \mathbf M_k, \mathbf X-\mathbf X_k\rangle.
\end{align} 
\EndFor
\end{algorithmic} 
\end{algorithm}

The following lemma gives a descent inequality for the function value resulting from updates produced by the mixed-norm-based ball oracle of Algorithm \ref{alg:multi-sdm-linear}. Its proof is deferred to Section \ref{sec:pf-multi-ball}.

\begin{lemma}\label{lem:multi-ball-descent}
Suppose that Assumption \ref{asp:multi-basic} holds. Let
$\{(\mathbf X_k,\mathbf M_k)\}$ be generated by Algorithm
\ref{alg:multi-sdm-linear} with step size $\eta>0$, and let
$\{\Delta \mathbf M_k\}$ be defined in \eqref{def:Delta-not-multi}. Then,
\begin{align}\label{ineq:matrix-sign-descent-multi}
f(\mathbf X_{k+1})
&\le f(\mathbf X_k)
-\eta\|\nabla f(\mathbf X_k)\|_{{\bf v},{\bf q}}
+2\eta\|\Delta \mathbf M_k\|_{{\bf v},{\bf q}}
+\frac{\eta^2}{2}
\qquad \forall k\ge0.
\end{align}
\end{lemma}

The next lemma gives an upper bound on the estimation error $\{\Delta {\bf M}_k\}$ for the momentum directions generated by Algorithm \ref{alg:multi-sdm-linear}. Its proof is deferred to Section \ref{sec:pf-multi-ball}.

\begin{lemma}\label{lem:multi-key-relation}
Suppose that Assumption \ref{asp:multi-basic} holds. Let
$\{(\mathbf X_k,\mathbf M_k)\}$ be generated by Algorithm
\ref{alg:multi-sdm-linear} with input parameters $\eta>0$ and
$\theta\in(0,1]$, and let $\{(\Delta \mathbf M_k,\Delta \mathbf G_k)\}$ be
defined in \eqref{def:Delta-not-multi}. Then,
\begin{align}\label{ineq:akeyrelation-multi}
\|\Delta \mathbf M_k\|_{{\bf v},{\bf q}}
&\le (1-\theta)^k\|\Delta \mathbf M_0\|_{{\bf v},{\bf q}}
+\theta
\Big\|
\sum_{t=0}^{k-1}(1-\theta)^t\Delta \mathbf G_{k-t}
\Big\|_{{\bf v},{\bf q}}
+\eta\sum_{t=0}^{k-1}(1-\theta)^{t+1}\qquad\forall k\ge0.
\end{align}
\end{lemma}

The next lemma gives an upper bound on the sum of the expected estimation error over the momentum directions generated by Algorithm \ref{alg:multi-sdm-linear}. Its proof is deferred to Section \ref{sec:pf-multi-ball}.

\begin{lemma}\label{lem:multi-F-con}
Suppose that Assumption \ref{asp:multi-basic} holds with
$\beta\in[1,2]$ and $q_j\in[1,2]$ for all $j\in[J]$.
Let $\{(\mathbf X_k,\mathbf M_k)\}$ be generated by Algorithm
\ref{alg:multi-sdm-linear} with input parameters $\eta>0$, $\theta\in(0,1]$,
and $B\ge1$. Let $\{\Delta \mathbf M_k\}$ be defined in
\eqref{def:Delta-not-multi}. Define
\begin{align}\label{def:sigma-multi}
\sigma_{{\bf v},{\bf q},B}
&:=
\frac{1}{\sqrt B}
\left\|
\left(
L_1^{-1/2}\|V_1\|_{S_{q_1}},
\ldots,
L_J^{-1/2}\|V_J\|_{S_{q_J}}
\right)
\right\|_\beta.
\end{align}
Then, for all $K\ge1$,
\begin{align}\label{ineq:expec-upbd-multi-norm}
\sum_{k=0}^{K-1}
\E[\|\Delta \mathbf M_k\|_{{\bf v},{\bf q}}]
&\le \frac{\sigma_{{\bf v},{\bf q},B}}{\theta}
+\sqrt{\theta}K\sigma_{{\bf v},{\bf q},B}
+\frac{\eta K}{\theta}.
\end{align}
\end{lemma}

The following theorem provides an upper bound on the expected $S_q$-norm-based stationarity achieved by Algorithm \ref{alg:multi-sdm-linear}. Its proof is deferred to Section \ref{sec:pf-multi-ball}.

\begin{theorem}\label{thm:multi-F-con}
Suppose that Assumption \ref{asp:multi-basic} holds with
$\beta\in[1,2]$ and $q_j\in[1,2]$ for all $j\in[J]$.
Let $K$ be the maximum iteration number for running Algorithm
\ref{alg:multi-sdm-linear}, and assume
$\Delta_f:=f(\mathbf X_0)-f_{\mathrm{low}}>0$. Let $\{\mathbf X_k\}$ be the iterates generated by
Algorithm \ref{alg:multi-sdm-linear} with input parameters
\begin{align}\label{def:eta-theta-F-multi}
\eta
=\sqrt{\frac{\Delta_f\theta}{K}},
\qquad
\theta
=\min\left\{
\frac{1}{\sigma_{{\bf v},{\bf q},B}}
\sqrt{\frac{\Delta_f}{K}},
1
\right\},
\end{align}
where $\sigma_{{\bf v},{\bf q},B}$ is defined in \eqref{def:sigma-multi}. 
Then, for all $K\ge1$, it holds that
\begin{align}\label{ineq:ave-stat-F-multi}
\frac{1}{K}
\sum_{k=0}^{K-1}
\E[\|\nabla f(\mathbf X_k)\|_{{\bf v},{\bf q}}]
&\le
6\sqrt{\sigma_{{\bf v},{\bf q},B}}
\left(\frac{\Delta_f}{K}\right)^{1/4}
+8\sqrt{\frac{\Delta_f}{K}}
+\frac{2\sigma_{{\bf v},{\bf q},B}^2}{\sqrt{\Delta_f K}}.
\end{align}
\end{theorem}

\subsection{Multi-block coupling with regularization oracle}\label{sec:multi-reg}

In this subsection, we propose and analyze a multi-block method with regularization oracle. This method is a multi-block extension of the $S_p$-$S_q$ method with regularization oracle (Algorithm \ref{alg:sdm-quadratic}) for solving problem \eqref{pb:matrix-opt-multi}, where the $S_p$-norm used to construct the regularization oracle is replaced by the mixed norm $\|\cdot\|_{{\bf w},{\bf p}}$. Details of this method are presented in Algorithm \ref{alg:multi-sdm-quadratic}.

\begin{algorithm}[!htbp] 
\caption{A multi-block method with regularization oracle} 
\label{alg:multi-sdm-quadratic} 
\begin{algorithmic}[0] 
\State \textbf{Input:} starting iterate $\mathbf X_0\in\MBspace$, momentum parameter $\theta\in(0,1]$, norm parameters $\alpha\in[2,\infty]$ and $\mathbf p\in[2,\infty]^J$, step size $\eta>0$, batch size $B\ge1$. 
\For{$k=0,1,2,\ldots$} 
\State Sample a mini-batch $\{\xi_{k,i}\}_{i=1}^B$ and compute $\bar{\mathbf G}_k=\frac{1}{B}\sum_{i=1}^B \mathbf G(\mathbf X_k;\xi_{k,i})$.
\State Compute the search direction: 
\begin{align}\label{update-mk-multi-2} 
\mathbf M_k = \begin{cases}
(1 - \theta) \mathbf M_{k-1} + \theta \bar{\mathbf G}_k, &\text{if } k \geq 1,\\
\bar{\mathbf G}_0, & \text{if } k =0.
\end{cases}
\end{align}
\State Update the next iterate:
\begin{align}\label{eq:update-xk-multi-2}
\mathbf X_{k+1}
= \argmin_{\mathbf X\in\MBspace}
\bigg\{
\langle \mathbf M_k, \mathbf X-\mathbf X_k\rangle
+ \frac{1}{2\eta}\|\mathbf X-\mathbf X_k\|_{{\bf w},{\bf p}}^2
\bigg\}.
\end{align} 
\EndFor
\end{algorithmic} 
\end{algorithm}  

We now make a few remarks on the solution of the quadratically regularized subproblem \eqref{eq:update-xk-multi-2}. Similar to \eqref{eq:update-xk-multi-1}, one valid evaluation of \eqref{eq:update-xk-multi-2} can always be obtained by rescaling the solutions of $J$ separate $S_{p_j}$-norm regularized problems, one for each block. The weighted $\ell_\alpha$-norm structure in $\|\cdot\|_{{\bf w},{\bf p}}$ only affects the step size for each block. For illustration, we consider the cases of the weighted $\ell_\alpha$-norm separately for $\alpha=2$ and $\alpha\in(2,\infty]$:
\begin{itemize}[leftmargin = .6cm]
    \item $\ell_2$-norm: In this case, \eqref{eq:update-xk-multi-2} employs the mixed norm $\|\cdot\|_{{\bf w},{\bf p}}$ with ${\bf w}=(2,L_1,\ldots,L_J)$. By the definition of $\|\cdot\|_{{\bf w},{\bf p}}$, \eqref{eq:update-xk-multi-2} is immediately separable. Consequently, we obtain
    \begin{align*}
    [X_{k+1}]_j = \argmin_{[X]_j}\left\{
    \langle [M_k]_j, [X]_j- [X_k]_j\rangle + \frac{L_j}{2\eta}\|[X]_j- [X_k]_j\|_{S_{p_j}}^2\right\}\qquad  \forall j\in[J].
    \end{align*}
    Therefore, the $\ell_2$-norm coupling simply yields a step size of $L_j^{-1}\eta$ for the $j$th block, where $L_j$ is the block-wise Lipschitz constant associated with the $j$th block.
    
    \item $\ell_\alpha$-norm ($\alpha\in(2,\infty]$): In this case, although the resulting subproblem appears non-separable, Lemma \ref{lem:mult-block-subproblem} shows that its solution can nevertheless be obtained by solving $J$ block-wise subproblems. Specifically, solving \eqref{eq:update-xk-multi-2} is equivalent to solving the following $J$ subproblems:
    \begin{align*}
        [X_{k+1}]_j = \argmin_{[X]_j}\left\{
    \langle [M_k]_j, [X]_j- [X_k]_j\rangle + \frac{L_j}{2\tilde\omega_j\eta}\|[X]_j- [X_k]_j\|_{S_{p_j}}^2\right\}\qquad  \forall j\in[J],
    \end{align*}
    where $\tilde\omega_j$'s are rescaling factors given by
    \begin{align*}
    \tilde\omega_j := \left( \frac{L_j^{-1/2}\|[M]_j\|_{S_{q_j}}}{\|{\bf M}\|_{{\bf v},{\bf q}}}\right)^{\beta-2}\qquad \forall j\in[J],
\end{align*}
    and the vector $(\tilde\omega_1,\ldots,\tilde\omega_J)$ satisfies $\|(\tilde\omega_1,\ldots,\tilde\omega_J)\|_{\beta/(\beta-2)}=1$. Thus, $\ell_\alpha$-norm-based coupling yields a step size of $L_j^{-1}\tilde\omega_j\eta$, where the rescaling factors $\{\tilde\omega_j\}_{j=1}^J$ form a unit vector in terms of $\ell_{\beta/(\beta-2)}$-norm, and $\tilde\omega_j$ is smaller when $\|[M]_j\|_{S_{q_j}}$ is larger.
\end{itemize}

The following lemma provides a descent inequality for the function value when updating with the mixed-norm-based regularization oracle of Algorithm \ref{alg:multi-sdm-quadratic}. Its proof is deferred to Section \ref{sec:pf-multi-reg}.

\begin{lemma}\label{lem:multi-qo-descent}
Suppose that Assumption \ref{asp:multi-basic} holds. Let
$\{(\mathbf X_k,\mathbf M_k)\}$ be generated by Algorithm
\ref{alg:multi-sdm-quadratic} with step size $\eta\in(0,1]$, and let
$\{\Delta \mathbf M_k\}$ be defined in \eqref{def:Delta-not-multi}. Then, for
any $k\ge0$,
\begin{align}
f(\mathbf X_{k+1})
&\le f(\mathbf X_k)
-\left(\frac{1}{4\eta}-\frac{1}{4}\right)
\|\mathbf X_{k+1}-\mathbf X_k\|_{{\bf w},{\bf p}}^2 \nonumber\\
&\quad
-\frac{\eta}{8}(1-\eta)
\|\nabla f(\mathbf X_k)\|_{{\bf v},{\bf q}}^2
+\frac{\eta}{4}(3-\eta)
\|\Delta \mathbf M_k\|_{{\bf v},{\bf q}}^2.
\label{ineq:matrix-sign-descent-qo-multi}
\end{align}
\end{lemma}

The next lemma provides an upper bound on the squared estimation error $\{\Delta M_k\}$ for the momentum directions generated by Algorithm \ref{alg:multi-sdm-quadratic}. Its proof is relegated to Section \ref{sec:pf-multi-reg}.

\begin{lemma}\label{lem:multi-square-deter}
Suppose that Assumption \ref{asp:multi-basic} holds. Let
$\{(\mathbf X_k,\mathbf M_k)\}$ be generated by Algorithm
\ref{alg:multi-sdm-quadratic} with input parameters $\eta>0$ and
$\theta\in(0,1]$, and let $\{(\Delta \mathbf M_k,\Delta \mathbf G_k)\}$ be
defined in \eqref{def:Delta-not-multi}. Then,
\begin{align}\label{ineq:square-multi-deter}
\|\Delta \mathbf M_k\|_{{\bf v},{\bf q}}^2
&\le
3(1-\theta)^{2k}\|\Delta \mathbf M_0\|_{{\bf v},{\bf q}}^2
+3\theta^2
\Big\|\sum_{t=0}^{k-1}(1-\theta)^t\Delta \mathbf G_{k-t}\Big\|_{{\bf v},{\bf q}}^2
\nonumber\\
&\quad
+\frac{3(1-\theta)^2}{\theta}
\sum_{t=0}^{k-1}(1-\theta)^t
\|\mathbf X_{k-t-1}-\mathbf X_{k-t}\|_{{\bf w},{\bf p}}^2.
\end{align}
\end{lemma}

The next lemma gives an upper bound on the sum of the expected squared estimation error over the momentum directions generated by Algorithm \ref{alg:multi-sdm-quadratic}. Its proof is relegated to Section \ref{sec:pf-multi-reg}.

\begin{lemma}\label{lem:multi-qo-error}
Suppose that Assumption \ref{asp:multi-basic} holds with
$\beta\in[1,2]$ and $q_j\in[1,2]$ for all $j\in[J]$.
Let $\{(\mathbf X_k,\mathbf M_k)\}$ be generated by Algorithm
\ref{alg:multi-sdm-quadratic} with input parameters $\eta>0$,
$\theta\in(0,1]$, and $B\ge1$. Let $\{\Delta \mathbf M_k\}$ be defined in
\eqref{def:Delta-not-multi}. Then, for all $K\ge1$,
\begin{align}\label{ineq:expec-upbd-spec-norm-multi}
\sum_{k=0}^{K-1}
\E[\|\Delta \mathbf M_k\|_{{\bf v},{\bf q}}^2]
&\le
\frac{3}{\theta}\sigma_{{\bf v},{\bf q},B}^2
+3\theta K\sigma_{{\bf v},{\bf q},B}^2
+\frac{3}{\theta^2}
\sum_{k=1}^{K-1}
\E[\|\mathbf X_k-\mathbf X_{k-1}\|_{{\bf w},{\bf p}}^2],
\end{align}
where $\sigma_{{\bf v},{\bf q},B}$ is defined in \eqref{def:sigma-multi}.
\end{lemma}

The following theorem provides an upper bound on the expected squared mixed-norm-based stationarity achieved by Algorithm \ref{alg:multi-sdm-quadratic}. Its proof is deferred to Section \ref{sec:pf-multi-reg}.

\begin{theorem}\label{thm:multi-spec-con}
Suppose that Assumption \ref{asp:multi-basic} holds with
$\beta\in[1,2]$ and $q_j\in[1,2]$ for all $j\in[J]$.
Let $K$ be the maximum iteration number for running Algorithm
\ref{alg:multi-sdm-quadratic}, and let
$\Delta_f:=f(\mathbf X_0)-f_{\mathrm{low}}$. Let $\{\mathbf X_k\}$ be the
iterates generated by Algorithm \ref{alg:multi-sdm-quadratic} with input
parameters
\begin{align}\label{def:eta-theta-spec-multi}
\eta=\frac{\theta}{4},
\qquad
\theta
=
\min\left\{
\sqrt{\frac{\Delta_f+\sigma_{{\bf v},{\bf q},B}^2}
{\sigma_{{\bf v},{\bf q},B}^2K}},
1
\right\},
\end{align}
where $\sigma_{{\bf v},{\bf q},B}$ is defined in \eqref{def:sigma-multi}. Then, for all $K\ge1$,
\begin{align}\label{ineq:ave-stat-spec-multi}
\frac{1}{K}\sum_{k=0}^{K-1}
\E[\|\nabla f(\mathbf X_k)\|_{{\bf v},{\bf q}}^2]
&\le
67\left[
\sqrt{
\frac{\sigma_{{\bf v},{\bf q},B}^2
(\Delta_f+\sigma_{{\bf v},{\bf q},B}^2)}{K}}
+\frac{\Delta_f+\sigma_{{\bf v},{\bf q},B}^2}{K}
\right].
\end{align}
\end{theorem}


\section{Proof of the main results}
In this section, we provide proofs of our main results in Sections \ref{sec:oracle}-\ref{sec:multi}.

\subsection{Proof of the main results in Section \ref{sec:oracle}}

\begin{proof}[\textbf{Proof of Lemma \ref{lem:single-block-oracle}}]


If $M=0$, the ball oracle is set-valued and we take the selection
$\Delta_p^{\rm b}(0)=0$ and the regularization oracle is uniquely minimized at $0$. For the rest of the proof, we assume that $M\neq 0$ and, without loss of generality, that $m\le n$.

Let $M=U\operatorname{diag}(\sigma)V^T$ be the singular value decomposition,
where $\sigma=(\sigma_1,\ldots,\sigma_m)\in\mathbb R_+^m$. Define
$Z_{\rm b}:=-U\operatorname{diag}(z)V^T$ as follows. If $q>1$, let
\(
z_i:={\sigma_i^{q-1}}/{\|\sigma\|_q^{q-1}}
\)
for all $i$. If $q=1$ (and hence $p=\infty$), let
$z_i=1$ whenever $\sigma_i>0$ and $z_i=0$ whenever $\sigma_i=0$.
Then $\|Z_{\rm b}\|_{S_p}=1$, so $Z_{\rm b}$ is feasible for the ball oracle. Moreover,
\[
\langle M,Z_{\rm b}\rangle=
-\|\sigma\|_q =-\|M\|_{S_q}.
\]
On the other hand, the trace H\"older inequality gives
\[
\langle M,Z\rangle\ge -\|M\|_{S_q}\|Z\|_{S_p}\ge -\|M\|_{S_q},
\qquad
\forall Z\ \text{with}\ \|Z\|_{S_p}\le1.
\]
Thus $Z_{\rm b}$ is a minimizer of the ball oracle, as given in \eqref{def:Delta1M-Delta2M}.

We next derive the formula for $\Delta_p^{\rm r}(M)$. For any $Z\neq0$, write
$Z=tW$ with $t=\|Z\|_{S_p}$ and $\|W\|_{S_p}=1$. The case $Z=0$ corresponds
to $t=0$. Then
\[
\langle M,Z\rangle+\frac12\|Z\|_{S_p}^2
\ge -t\|M\|_{S_q}+\frac12t^2,
\]
and equality is attained by choosing $W=Z_{\rm b}$. The one-dimensional problem
$\min_{t\ge0}\{-t\|M\|_{S_q}+\frac12t^2\}$ has the solution
$t^*=\|M\|_{S_q}$. Hence a minimizer of the regularization oracle is also as given in \eqref{def:Delta1M-Delta2M}.
\end{proof}

\begin{proof}[\textbf{Proof of Lemma \ref{lem:mult-block-subproblem}}]
If $\|{\bf M}\|_{{\bf v},{\bf q}}=0$, then ${\bf M}=0$. The ball oracle is
set-valued and contains the origin, while the regularization oracle is uniquely
minimized at the origin. Hence the stated selections are valid.

Assume now that $\|{\bf M}\|_{{\bf v},{\bf q}}>0$, so at least one block is nonzero. For any feasible
${\bf Z}$, the duality of the mixed norms gives
\begin{align}\label{ineq:Holder-mix}
\langle{\bf M},{\bf Z}\rangle
\ge -\|{\bf M}\|_{{\bf v},{\bf q}}\|{\bf Z}\|_{{\bf w},{\bf p}}
\ge -\|{\bf M}\|_{{\bf v},{\bf q}}.    
\end{align}
Let ${\bf Z}_{\rm b}$ be such that
\begin{align*}
[{\bf Z}_{\rm b}]_j = L_j^{-1/2}
\left(
\frac{L_j^{-1/2}\|[M]_j\|_{S_{q_j}}}
{\|{\bf M}\|_{{\bf v},{\bf q}}}
\right)^{\beta-1}
\Delta_{p_j}^{\rm b}([M]_j)\qquad\forall j\in[J].
\end{align*}
It remains to show that ${\bf Z}_{\rm b}$ is a solution to the ball oracle. By Lemma \ref{lem:single-block-oracle}, one has $\|\Delta_{p_j}^{\rm b}([M]_j)\|_{S_{p_j}}=1$ when $[M]_j\neq0$, while $\Delta_{p_j}^{\rm b}([M]_j)=0$ when $[M]_j=0$. Therefore the zero blocks do
not contribute to either the norm or the inner product. If $\beta>1$, using
$\alpha(\beta-1)=\beta$ gives
\[
\begin{aligned}
\|{\bf Z}_{\rm b}\|_{{\bf w},{\bf p}}^\alpha
&=
\sum_{j=1}^J
\left(
\frac{L_j^{-1/2}\|[M]_j\|_{S_{q_j}}}
{\|{\bf M}\|_{{\bf v},{\bf q}}}
\right)^{\alpha(\beta-1)}
=
\frac{
\sum_{j=1}^J
\left(L_j^{-1/2}\|[M]_j\|_{S_{q_j}}\right)^\beta
}{
\|{\bf M}\|_{{\bf v},{\bf q}}^\beta
}
=1.
\end{aligned}
\]
If $\beta=1$ and $\alpha=\infty$, then 
\[
\|{\bf Z}_{\rm b}\|_{{\bf w},{\bf p}}
=
\max_{1\le j\le J}
\|\Delta_{p_j}^{\rm b}([M]_j)\|_{S_{p_j}}
=1.
\]
It then follows that ${\bf Z}_{\rm b}$ is feasible for the ball oracle. Using the single-block identity
$\langle [M]_j,\Delta_{p_j}^{\rm b}([M]_j)\rangle=-\|[M]_j\|_{S_{q_j}}$,
we obtain that, when $\beta\geq1$,
\[
\begin{aligned}
\left\langle{\bf M},{\bf Z}_{\rm b}\right\rangle
&=
-\sum_{j=1}^J
L_j^{-1/2}\|[M]_j\|_{S_{q_j}}
\left(
\frac{L_j^{-1/2}\|[M]_j\|_{S_{q_j}}}
{\|{\bf M}\|_{{\bf v},{\bf q}}}
\right)^{\beta-1}
\\
&=
-\frac{
\sum_{j=1}^J
\left(L_j^{-1/2}\|[M]_j\|_{S_{q_j}}\right)^\beta
}{
\|{\bf M}\|_{{\bf v},{\bf q}}^{\beta-1}
}
=-\|{\bf M}\|_{{\bf v},{\bf q}}.
\end{aligned}
\]
This, together with \eqref{ineq:Holder-mix}, implies that \eqref{def:multi-ball-oracle} is a valid solution to the ball oracle.

For the regularization oracle, let ${\bf Z}=t{\bf W}$ with
$t=\|{\bf Z}\|_{{\bf w},{\bf p}}$ and
$\|{\bf W}\|_{{\bf w},{\bf p}}\le1$. Then
\[
\langle{\bf M},{\bf Z}\rangle+\frac12\|{\bf Z}\|_{{\bf w},{\bf p}}^2
\ge -t\|{\bf M}\|_{{\bf v},{\bf q}}+\frac12t^2.
\]
Equality is attained by taking
${\bf W}={\bf Z}_{\rm b}$ and
$t=\|{\bf M}\|_{{\bf v},{\bf q}}$. Hence the selected regularization-oracle
minimizer
$\|{\bf M}\|_{{\bf v},{\bf q}}{\bf Z}_{\rm b}$ is valid,
which proves
\eqref{def:multi-reg-oracle}.
\end{proof}

\subsection{Proof of the main results in Section \ref{sec:single}}

In this subsection, we provide proofs of our results in Section \ref{sec:single}.

\subsubsection{Proof of the main results in Section \ref{subsec:mmi-single}}\label{sec:pf-single-mmi}

\begin{proof}[\textbf{Proof of Lemma \ref{lem:nuc-ppt}}]
The case $q=2$ follows from
\begin{align*}
\|A\|_{S_2}^2=\mathrm{Tr}(AA^T)\le \|\Lambda\|_{S_\infty}\mathrm{Tr}(AA^T\Lambda^{-1}).
\end{align*}
We now consider $q\in[1,2)$.
First, suppose $\Lambda$ is diagonal, i.e., $\Lambda=\operatorname{diag}(\lambda_1,\ldots,\lambda_m)$, where $\lambda_i > 0$. Then, it holds that 
\begin{align*}
\|A\|_{S_q}^q = \mathrm{Tr}((AA^T)^{\frac{q}{2}}) \le \sum_{i=1}^m[AA^T]_{ii}^{\frac{q}{2}} \le \Big(\sum_{i=1}^m \lambda_i^{\frac{q}{2-q}}\Big)^{1-\frac{q}{2}}\Big(\sum_{i=1}^m[AA^T]_{ii}\lambda_i^{-1}\Big)^{\frac{q}{2}},
\end{align*}
where the first inequality is due to the concavity of $(\cdot)^{\frac{q}{2}}$, and the second inequality is due to H\"older inequality with exponents $2/q$ and $2/(2-q)$. It then follows that
\begin{align*}
\|A\|_{S_q}^2 \le \Big(\sum_{i=1}^m \lambda_i^{\frac{q}{2-q}}\Big)^{\frac{2-q}{q}}\Big(\sum_{i=1}^m[AA^T]_{ii}\lambda_i^{-1}\Big) = \|\Lambda\|_{S_{\frac{q}{2-q}}}\mathrm{Tr}(AA^T\Lambda^{-1}) = \|\Lambda\|_{S_{\frac{q}{2-q}}}\mathrm{Tr}(A^T\Lambda^{-1}A).  
\end{align*}
When $\Lambda$ is not diagonal, consider its eigen-decomposition, $\Lambda = U\Sigma U^T$, and let $\widetilde{A}:=U^TA$. Then, we have
\begin{align*}
\|A\|_{S_q}^2 = \|\widetilde{A}\|_{S_q}^2 \le \|\Sigma\|_{S_{\frac{q}{2-q}}}\mathrm{Tr}(\widetilde{A}^T\Sigma^{-1}\widetilde{A}) = \|\Lambda\|_{S_{\frac{q}{2-q}}}\mathrm{Tr}(A^T\Lambda^{-1}A) = \|\Lambda\|_{S_{\frac{q}{2-q}}}\mathrm{Tr}(AA^T\Lambda^{-1}),
\end{align*}
as desired.
\end{proof}

\begin{proof}[\textbf{Proof of Lemma \ref{lem:concentration-Sp}}]
By Lemma \ref{lem:nuc-ppt}, it holds that
\[
\Big\|\sum_{t=0}^{k-1}A_t\Big\|_{S_q}^2 \le \|\Lambda\|_{S_{\frac{q}{2-q}}}\mathrm{Tr}\bigg[\Big(\sum_{t=0}^{k-1}A_t \Big)\Big(\sum_{t=0}^{k-1}A_t\Big) ^T\Lambda^{-1}\bigg].
\]
Moreover, since $\{A_t\}$ is a sequence of martingale differences, for $i \neq j$, we have $\E[A_i A_j^T] = 0$.
Then, taking expectation over the previous inequality yields
\[
\E\Big[\Big\|\sum_{t=0}^{k-1}A_t\Big\|_{S_q}^2\Big] \leq \|\Lambda\|_{S_{\frac{q}{2-q}}}\mathrm{Tr}\bigg[\E \Big[\Big(\sum_{t=0}^{k-1}A_t \Big)\Big(\sum_{t=0}^{k-1}A_t\Big) ^T\Big]\Lambda^{-1}\bigg] = \|\Lambda\|_{S_{\frac{q}{2-q}}}\sum_{t=0}^{k-1}\mathrm{Tr}(\E[A_tA_t^T]\Lambda^{-1}).
\]
as desired.
\end{proof}

\subsubsection{Proof of the main results in Section \ref{sec:lmo}}\label{sec:pf-single-lmo}

\begin{proof}[\textbf{Proof of Lemma \ref{lem:desc-ineq-Sp-ball}}]
In view of the update of $\{X_k\}$ in \eqref{eq:update-xk-1}, one has $\|X_{k+1} - X_k\|_{S_p}=\eta$ and $\langle M_k,X_{k+1}-X_k\rangle=-\eta\|M_k\|_{S_q}$. Using these and \eqref{ineq:desc-lpq} with $(Y,X)=(X_{k+1},X_k)$, we obtain that
\begin{align*}
    f(X_{k+1}) & \le f(X_k) + \langle \nabla f(X_k), X_{k+1} - X_k \rangle + \frac{L_{p,q}}{2} \|X_{k+1} - X_k\|^2_{S_p} \\
    & = f(X_k) + \langle M_k, X_{k+1} - X_k \rangle + \langle \nabla f(X_k) - M_k, X_{k+1} - X_k \rangle + \frac{L_{p,q}\eta^2}{2} \\
    & = f(X_k) - \eta \|M_k\|_{S_q} + \langle \nabla f(X_k) - M_k, X_{k+1} - X_k \rangle + \frac{L_{p,q}\eta^2}{2} \\
    & \le f(X_k) - \eta \|M_k\|_{S_q} + \|\nabla f(X_k) - M_k\|_{S_q}\|X_{k+1} - X_k\|_{S_p} + \frac{L_{p,q}\eta^2}{2}\\
    & = f(X_k) - \eta \|M_k\|_{S_q} + \eta\|\nabla f(X_k) - M_k\|_{S_q} + \frac{L_{p,q}\eta^2}{2} \\
    & \le f(X_k) - \eta \|\nabla f(X_k)\|_{S_q} + 2\eta\|\nabla f(X_k) - M_k\|_{S_q} + \frac{L_{p,q}\eta^2}{2},
\end{align*}
where in the second inequality we utilize the trace H\"older inequality.
\end{proof}

\begin{proof}[\textbf{Proof of Lemma \ref{lem:momeutm-bd-Sp-ball}}]
The statement is trivial for $k=0$. For $k\ge1$, by the update of $\{M_k\}$ in \eqref{update-mk-1}, it holds that
\begin{align}
\Delta M_k & = (1-\theta)M_{k-1} + \theta \bar G_k - \nabla f(X_k)\nonumber\\
& = (1-\theta)\Delta M_{k-1} + \theta \Delta G_k + (1-\theta) (\nabla f(X_{k-1}) - \nabla f(X_k))\nonumber\\
& = (1-\theta)^k\Delta M_0 + \theta \sum_{t=0}^{k-1}(1-\theta)^t\Delta G_{k-t} - \sum_{t=0}^{k-1} (1-\theta)^{t+1}(\nabla f(X_{k-t}) - \nabla f(X_{k-t-1})). \label{eq:DeltaMk-idtt-alg1}
\end{align}
Taking the $S_q$-norm and invoking the triangle inequality and Assumption \ref{asp:basic}(b) yields the result.
\end{proof}

\begin{proof}[\textbf{Proof of Lemma \ref{lem:F-con}}]
Fix $k\ge1$. Applying Lemma \ref{lem:concentration-Sp} to the martingale
difference sequence
\[
A_{s-1}:=(1-\theta)^{k-s}\Delta G_s,\qquad s=1,\ldots,k,
\]
with $\Lambda$ being $\Lambda_\varepsilon=(VV^T+\varepsilon I)^{1-\frac{q}{2}} \succ 0$, it holds that for any $\varepsilon>0$,
\begin{align*}
\E\Big[\Big\|\sum_{t=0}^{k-1}(1-\theta)^t\Delta G_{k-t}\Big\|_{S_q}^2\Big]
&\le \|\Lambda_\varepsilon\|_{S_{\frac{q}{2-q}}}
\sum_{t=0}^{k-1}(1-\theta)^{2t}
\mathrm{Tr}\!\left(\E[\Delta G_{k-t}\Delta G_{k-t}^T]
\Lambda_\varepsilon^{-1}\right)\\
&\le
\|\Lambda_\varepsilon\|_{S_{\frac{q}{2-q}}}
\frac{1}{B}
\sum_{t=0}^{k-1}(1-\theta)^{2t}
\mathrm{Tr}\!\left(
(VV^T+\varepsilon I)\Lambda_\varepsilon^{-1}
\right)\\
& = \|\Lambda_\varepsilon\|_{S_{\frac{q}{2-q}}}
\frac{1}{B}
\sum_{t=0}^{k-1}(1-\theta)^{2t}
\mathrm{Tr}\!\left(
(VV^T+\varepsilon I)^{\frac{q}{2}}\right),
\end{align*}
where the second inequality is due to Assumption \ref{asp:basic}(c). Then, letting $\varepsilon\downarrow0$ yields
\begin{align}
\E\Big[\Big\|\sum_{t=0}^{k-1}(1-\theta)^t\Delta G_{k-t}\Big\|_{S_q}^2\Big]
&\le
\big(\mathrm{Tr}[(VV^T)^{\frac{q}{2}}]\big)^{\frac{2-q}{q}}
\frac{\mathrm{Tr}[(VV^T)^{\frac{q}{2}}]}{B}
\sum_{t=0}^{k-1}(1-\theta)^{2t}\nonumber\\
&=
\frac{\|V\|_{S_q}^2}{B}
\sum_{t=0}^{k-1}(1-\theta)^{2t}
\leq \frac{\sigma_{q,B}^2}{\theta}.\label{ineq:E-sum-Deltakt-alg1}
\end{align}
By Jensen's inequality, this further implies that
\begin{align*}
    \E\Big[\Big\|\sum_{t=0}^{k-1}(1-\theta)^t\Delta G_{k-t}\Big\|_{S_q}\Big] 
    &\le
    \frac{\sigma_{q,B}}{\sqrt{\theta}}.
\end{align*}
Moreover, $M_0=\bar G_0$ implies $\E[\|\Delta M_0\|_{S_q}]\le \sigma_{q,B}$.
Taking expectation of both sides of \eqref{ineq:akeyrelation}, we obtain that
\begin{align*}
\E[\|\Delta M_k\|_{S_q}] & \le (1-\theta)^k\E[\|\Delta M_0\|_{S_q}]  + \theta \E\Big[\Big\|\sum_{t=0}^{k-1}(1-\theta)^t\Delta G_{k-t}\Big\|_{S_q}\Big] + L_{p,q}\eta \sum_{t=0}^{k-1} (1-\theta)^{t+1}\\
& \le (1-\theta)^k \sigma_{q,B} + \sqrt{\theta}\sigma_{q,B} + \frac{L_{p,q}\eta}{\theta}.
\end{align*}
Summing this over $k=0,\ldots,K-1$, we obtain that 
\begin{align*}
\sum_{k=0}^{K-1}\E[\|\Delta M_k\|_{S_q}] \le \frac{\sigma_{q,B}}{\theta} + \sqrt{\theta} K \sigma_{q,B} + \frac{L_{p,q}\eta K}{\theta}.  
\end{align*}
\end{proof}

\begin{proof}[\textbf{Proof of Theorem \ref{thm:F-con}}]
In view of \eqref{ineq:matrix-sign-descent} and \eqref{ineq:expec-upbd-nuclear-norm}, we see that for all $K\ge1$,
\begin{align*}
\frac{1}{K}\sum_{k=0}^{K-1}\E[\|\nabla f(X_k)\|_{S_q}] & \le \frac{\Delta_f}{\eta K} + \frac{L_{p,q}\eta}{2} + \frac{2}{K}\sum_{k=0}^{K-1}\E[\|\Delta M_k\|_{S_q}]\\
&\le \frac{\Delta_f}{\eta K} + \frac{L_{p,q}\eta}{2} + \frac{2\sigma_{q,B}}{\theta K} + 2\sqrt{\theta}\sigma_{q,B} + \frac{2L_{p,q}\eta}{\theta} \\
&\le \frac{\Delta_f}{\eta K} + \frac{2\sigma_{q,B}}{\theta K} + 2\sqrt{\theta}\sigma_{q,B} + \frac{3L_{p,q}\eta}{\theta}.
\end{align*}
By the definition of $\eta$ in \eqref{def:eta-theta-F}, one has that for all $K\ge1$,
\begin{align*}
\frac{1}{K}\sum_{k=0}^{K-1}\E[\|\nabla f(X_k)\|_{S_q}] & \le 4\sqrt{\frac{\Delta_fL_{p,q}}{\theta K}} + 2\sqrt{\theta}\sigma_{q,B} + \frac{2\sigma_{q,B}}{\theta K}.
\end{align*}
In addition, in view of the definition of $\theta$ in \eqref{def:eta-theta-F}, we observe that $\sigma_{q,B}\le\sqrt{\frac{\Delta_fL_{p,q}}{K}}$ when $\theta=1$. This, along with the above inequality, implies that when $\theta=1$,
\begin{align*}
\frac{1}{K}\sum_{k=0}^{K-1}\E[\|\nabla f(X_k)\|_{S_q}] \le 8\sqrt{\frac{\Delta_fL_{p,q}}{K}}, \qquad\forall K\ge1.
\end{align*}
On the other hand, when $\theta<1$, we can derive that
\begin{align*}
    \frac{1}{K}\sum_{k=0}^{K-1}\E[\|\nabla f(X_k)\|_{S_q}] \le 6\sqrt{\sigma_{q,B}}\bigg(\frac{\Delta_fL_{p,q}}{K}\bigg)^{1/4} + \frac{2\sigma_{q,B}^2}{\sqrt{\Delta_fL_{p,q}K}}.
\end{align*}
Combining these two cases of $\theta$, we obtain that \eqref{ineq:ave-stat-F} holds as desired.
\end{proof}

\subsubsection{Proof of the main results in Section \ref{sec:qmo}}\label{sec:pf-single-qmo}

\begin{proof}[\textbf{Proof of Lemma \ref{lem:ineq-descent-qmo-single}}]
In view of the update of $\{X_k\}$ in \eqref{update-xk-2}, one has $\|X_{k+1}-X_k\|_{S_p}=\eta\|M_k\|_{S_q}$ and $\langle M_k, X_{k+1} - X_k\rangle=-\eta\|M_k\|_{S_q}^2$. Using these and \eqref{ineq:desc-lpq} with $(Y,X)=(X_{k+1},X_k)$, we obtain that
\begin{align*}
    f(X_{k+1}) 
    & \le f(X_k)
    + \langle \nabla f(X_k), X_{k+1} - X_k \rangle
    + \frac{L_{p,q}}{2}\|X_{k+1} - X_k\|^2_{S_p} \\
    & = f(X_k)
    + \langle M_k, X_{k+1} - X_k \rangle
    + \langle \nabla f(X_k) - M_k, X_{k+1} - X_k \rangle 
    + \frac{L_{p,q}}{2}\|X_{k+1} - X_k\|^2_{S_p} \\
    & = f(X_k)
    - \eta \|M_k\|_{S_q}^2
    + \langle \nabla f(X_k) - M_k, X_{k+1} - X_k \rangle
    + \frac{L_{p,q}}{2}\|X_{k+1} - X_k\|^2_{S_p} \\
    & \le f(X_k)
    - \eta \|M_k\|_{S_q}^2
    + \frac{\eta}{2}\|\nabla f(X_k) - M_k\|_{S_q}^2
    + \Big(\frac{1}{2\eta} + \frac{L_{p,q}}{2}\Big)
    \|X_{k+1} - X_k\|^2_{S_p}\\
    & = f(X_k)
    - \Big(\frac{1}{2\eta} - \frac{L_{p,q}}{2}\Big)
    \|X_{k+1} - X_k\|_{S_p}^2
    + \frac{\eta}{2}\|\nabla f(X_k) - M_k\|_{S_q}^2\\
    & = f(X_k)
    - \Big(\frac{1}{4\eta} - \frac{L_{p,q}}{4}\Big)
    \|X_{k+1} - X_k\|_{S_p}^2
    - \frac{\eta}{4}(1 - L_{p,q}\eta)\|M_k\|_{S_q}^2 
    + \frac{\eta}{2}\|\nabla f(X_k) - M_k\|_{S_q}^2 \\
    & \le f(X_k)
    - \Big(\frac{1}{4\eta} - \frac{L_{p,q}}{4}\Big)
    \|X_{k+1} - X_k\|_{S_p}^2 
    - \frac{\eta}{8}(1 - L_{p,q}\eta)
    \|\nabla f(X_k)\|_{S_q}^2 \\
    &\qquad
    + \frac{\eta}{4}(3 - L_{p,q}\eta)
    \|\nabla f(X_k) - M_k\|_{S_q}^2,
\end{align*}
where the second inequality follows from $\langle A,B\rangle\le\|A\|_{S_q}\|B\|_{S_p}\le\frac{\eta}{2}\|A\|_{S_q}^2+\frac{1}{2\eta}\|B\|_{S_p}^2$ for all $A,B\in\R^{m\times n}$, and the last inequality follows from $-\|A\|_{S_q}^2\le -\frac{1}{2}\|A+B\|_{S_q}^2 + \|B\|_{S_q}^2$ for all $A,B\in\R^{m\times n}$. Hence, \eqref{ineq:matrix-sign-descent-qo} holds as desired.
\end{proof}

\begin{proof}[{\textbf{Proof of Lemma \ref{lem:momeutm-bd-Sp-reg}}}]
The statement for $k=0$ is trivial. We next consider the case $k\ge1$. By the update of $\{M_k\}$ in \eqref{update-mk-2} and the same arguments used to prove \eqref{eq:DeltaMk-idtt-alg1}, for all $k\ge1$, it holds that
\begin{align*}
\Delta M_k
&= (1-\theta)^k\Delta M_0
+ \theta \sum_{t=0}^{k-1}(1-\theta)^t\Delta G_{k-t} 
- (1-\theta) \sum_{t=0}^{k-1}(1-\theta)^t
(\nabla f(X_{k-t}) - \nabla f(X_{k-t-1})).
\end{align*}
Taking the $S_q$-norm and invoking the triangle inequality and Assumption \ref{asp:basic}(b) yields
\begin{align*}
\|\Delta M_k\|_{S_q}
&\le (1-\theta)^k\|\Delta M_0\|_{S_q}
+ \theta
\Big\|\sum_{t=0}^{k-1}(1-\theta)^t\Delta G_{k-t}\Big\|_{S_q} \\
&\qquad
+ (1-\theta) L_{p,q}
\sum_{t=0}^{k-1}(1-\theta)^t
\|X_{k-t-1} - X_{k-t}\|_{S_p}.   
\end{align*}
Squaring both sides and using the inequality $(a+b+c)^2\le3(a^2+b^2+c^2)$ for all $a,b,c>0$, we obtain
\begin{align}
\|\Delta M_k\|_{S_q}^2
& \le 3(1-\theta)^{2k}\|\Delta M_0\|_{S_q}^2
+ 3\theta^2
\Big\|\sum_{t=0}^{k-1}(1-\theta)^t\Delta G_{k-t}\Big\|_{S_q}^2 \nonumber\\
&\qquad
+ 3(1-\theta)^2 L_{p,q}^2
\bigg(
\sum_{t=0}^{k-1}(1-\theta)^t
\|X_{k-t-1} - X_{k-t}\|_{S_p}
\bigg)^2.\label{ineq:iter-DeltaMk2}
\end{align}
By the convexity of $(\cdot)^2$, one has that
\begin{align*}
\bigg(
\frac{
\sum_{t=0}^{k-1}(1-\theta)^t
\|X_{k-t-1} - X_{k-t}\|_{S_p}
}{
\sum_{t=0}^{k-1}(1-\theta)^t
}
\bigg)^2
\le
\frac{
\sum_{t=0}^{k-1}(1-\theta)^t
\|X_{k-t-1} - X_{k-t}\|_{S_p}^2
}{
\sum_{t=0}^{k-1}(1-\theta)^t
}, 
\end{align*}
which implies that
\begin{align*}
\Big(
\sum_{t=0}^{k-1}(1-\theta)^t
\|X_{k-t-1} - X_{k-t}\|_{S_p}
\Big)^2
&\le
\Big(\sum_{t=0}^{k-1}(1-\theta)^t\Big)
\Big(
\sum_{t=0}^{k-1}(1-\theta)^t
\|X_{k-t-1} - X_{k-t}\|_{S_p}^2
\Big) \\
&\le \frac{1}{\theta}
\sum_{t=0}^{k-1}(1-\theta)^t
\|X_{k-t-1} - X_{k-t}\|_{S_p}^2.
\end{align*}
This along with \eqref{ineq:iter-DeltaMk2} implies the desired result.
\end{proof}

\begin{proof}[\textbf{Proof of Lemma \ref{lem:F-con-reg}}]
Fix $k\ge1$. Applying Lemma \ref{lem:concentration-Sp} to the martingale
difference sequence
\[
A_{s-1}:=(1-\theta)^{k-s}\Delta G_s,\qquad s=1,\ldots,k,
\]
with $\Lambda$ being $\Lambda_\varepsilon=(VV^T+\varepsilon I)^{1-\frac{q}{2}} \succ 0$ and
letting $\varepsilon\downarrow0$, we obtain, by the same arguments used to prove \eqref{ineq:E-sum-Deltakt-alg1}, that
\begin{align*}
\E\Big[\Big\|\sum_{t=0}^{k-1}(1-\theta)^t\Delta G_{k-t}\Big\|_{S_q}^2\Big]
& \le \sigma_{q,B}^2\sum_{t=0}^{k-1}(1-\theta)^{2t}
\le \frac{\sigma_{q,B}^2}{\theta}.
\end{align*}
Moreover, $M_0=\bar G_0$ implies $\E[\|\Delta M_0\|_{S_q}^2]\le \sigma_{q,B}^2$. Using these and taking expectation of \eqref{ineq:square-Sp-deter}, we obtain that
\begin{align*}
\E[\|\Delta M_k\|_{S_q}^2]
& \le 3(1-\theta)^{2k}\E[\|\Delta M_0\|_{S_q}^2] 
+ 3\theta^2
\E\Big[\Big\|\sum_{t=0}^{k-1}(1-\theta)^t
\Delta G_{k-t}\Big\|_{S_q}^2\Big] \\
&\qquad
+ \frac{3(1-\theta)^2 L_{p,q}^2}{\theta}
\sum_{t=0}^{k-1}(1-\theta)^t
\E[\|X_{k-t-1} - X_{k-t}\|_{S_p}^2]\\
&\le 3(1-\theta)^{2k}\sigma_{q,B}^2 + 3\theta \sigma_{q,B}^2\\
&\qquad
+ \frac{3(1-\theta)^2 L_{p,q}^2}{\theta}
\sum_{t=0}^{k-1}(1-\theta)^t
\E[\|X_{k-t-1} - X_{k-t}\|_{S_p}^2].
\end{align*}
Summing this inequality over $k=0,\ldots,K-1$ and using $\theta\in(0,1]$, we obtain that \eqref{ineq:expec-upbd-spec-norm} holds.
\end{proof}

\begin{proof}[\textbf{Proof of Theorem \ref{thm:spec-con}}]
Taking expectation over \eqref{ineq:matrix-sign-descent-qo}, telescoping over $k=0,\ldots,K-1$, and using \eqref{ineq:expec-upbd-spec-norm}, we have
\begin{align*}
&\frac{1 - L_{p,q}\eta}{8K}
\sum_{k=0}^{K-1}\E[\|\nabla f(X_k)\|_{S_q}^2]
\\
\le&\, \frac{\Delta_f}{K\eta}
- \frac{1}{K\eta}
\Big(\frac{1}{4\eta} - \frac{L_{p,q}}{4}\Big)
\sum_{k=0}^{K-1}
\E[\|X_{k+1} - X_k\|_{S_p}^2]  + \frac{3 - L_{p,q}\eta}{4K}\sum_{k=0}^{K-1}\E[\|\Delta M_k\|_{S_q}^2] \\
\le&\, \frac{\Delta_f}{K\eta}
- \bigg[
\frac{1}{K\eta}
\Big(\frac{1}{4\eta} - \frac{L_{p,q}}{4}\Big)
- \frac{3L_{p,q}^2}{\theta^2}
\cdot\frac{3-L_{p,q}\eta}{4K}
\bigg]
\sum_{k=0}^{K-1}
\E[\|X_{k+1} - X_k\|_{S_p}^2] \\
&\quad + \frac{3(3 - L_{p,q}\eta)}{4K\theta}\sigma_{q,B}^2  + \frac{3(3 - L_{p,q}\eta)\theta}{4}\sigma_{q,B}^2.
\end{align*}
With $\eta=\theta/(4L_{p,q})$ and $\theta\le1$, the coefficient of the second term is nonnegative and $1-L_{p,q}\eta\ge3/4$. Therefore,
\begin{align*}
\frac{1}{K}\sum_{k=0}^{K-1}\E[\|\nabla f(X_k)\|_{S_q}^2]
&\le \frac{128L_{p,q}\Delta_f}{3\theta K}+\frac{24\sigma_{q,B}^2}{\theta K}+24\theta\sigma_{q,B}^2\\
&\le \frac{43(L_{p,q}\Delta_f+\sigma_{q,B}^2)}{\theta K}+24\theta\sigma_{q,B}^2.
\end{align*}
If $\theta<1$, the definition of $\theta$ gives
\begin{align*}
\frac{1}{K}\sum_{k=0}^{K-1}\E[\|\nabla f(X_k)\|_{S_q}^2]
\le 67\sqrt{\frac{\sigma_{q,B}^2(L_{p,q}\Delta_f+\sigma_{q,B}^2)}{K}}.
\end{align*}
If $\theta=1$, then $\sigma_{q,B}^2\le (L_{p,q}\Delta_f+\sigma_{q,B}^2)/K$, and hence
\begin{align*}
\frac{1}{K}\sum_{k=0}^{K-1}\E[\|\nabla f(X_k)\|_{S_q}^2]
\le 67\frac{L_{p,q}\Delta_f+\sigma_{q,B}^2}{K}.
\end{align*}
Combining the two cases yields \eqref{ineq:ave-stat-spec}.
\end{proof}

\subsection{Proof of the main results in Section \ref{sec:multi}}

In this subsection, we provide proofs of our results in Section \ref{sec:multi}.

\subsubsection{Proof of the main results in Section \ref{sec:multi-mii}}\label{sec:pf-multi-mii}

\begin{proof}[\textbf{Proof of Lemma \ref{lem:multi-block-martingale}}]
For convenience, define
\[
\mathbf S_k := \sum_{t=0}^{k-1}\mathbf A_t
\quad \text{and} \quad
[S_k]_j:=\sum_{t=0}^{k-1}[A_t]_j.
\]
By Lemma \ref{lem:nuc-ppt}, for each $j=1,\ldots,J$, we have
\[
L_j^{-1}\|[S_k]_j\|_{S_{q_j}}^2
\le
L_j^{-1}
\|\Lambda_j\|_{S_{\frac{q_j}{2-q_j}}}
\operatorname{Tr}([S_k]_j[S_k]_j^T\Lambda_j^{-1})
=: Y_j^2.
\]
Then, it follows that
\[
\|\mathbf S_k\|_{{\bf v},{\bf q}}^2
=
\Big(
\sum_{j=1}^J
\Big(
L_j^{-1/2}\|[S_k]_j\|_{S_{q_j}}
\Big)^\beta
\Big)^{\frac 2\beta}
\le
\Big(
\sum_{j=1}^J Y_j^\beta
\Big)^{\frac 2\beta}.
\]
Since the martingale-difference property of $\{[A_t]_j\}$ directly implies
$\E[Y_j^2]=B_j^2$, taking expectation and applying Minkowski's inequality in
$L_{\frac 2\beta}$ gives
\begin{align*}
\E[\|\mathbf S_k\|_{{\bf v},{\bf q}}^2]
&\le
\E\Big[
\Big(
\sum_{j=1}^J
Y_j^\beta
\Big)^{\frac 2\beta}
\Big]
\le
\Big[
\sum_{j=1}^J
(\E[Y_j^2])^{\frac\beta 2}
\Big]^{\frac 2\beta}
=
\Big[
\sum_{j=1}^J
B_j^\beta
\Big]^{\frac 2\beta}.
\end{align*}
This completes the proof.
\end{proof}

\subsubsection{Proof of the main results in Section \ref{sec:multi-ball}}\label{sec:pf-multi-ball}

\begin{proof}[\textbf{Proof of Lemma \ref{lem:multi-ball-descent}}]
By the update \eqref{eq:update-xk-multi-1} and the duality of
$\|\cdot\|_{{\bf w},{\bf p}}$ and $\|\cdot\|_{{\bf v},{\bf q}}$,
\[
\langle \mathbf M_k,\mathbf X_{k+1}-\mathbf X_k\rangle
=-\eta\|\mathbf M_k\|_{{\bf v},{\bf q}},
\qquad
\|\mathbf X_{k+1}-\mathbf X_k\|_{{\bf w},{\bf p}}=\eta.
\]
Using \eqref{ineq:desc-multi} with
$(\mathbf Y,\mathbf X)=(\mathbf X_{k+1},\mathbf X_k)$ gives
\begin{align*}
f(\mathbf X_{k+1})
&\le f(\mathbf X_k)
+\langle\nabla f(\mathbf X_k),\mathbf X_{k+1}-\mathbf X_k\rangle
+\frac{1}{2}
\|\mathbf X_{k+1}-\mathbf X_k\|_{{\bf w},{\bf p}}^2
\\
&= f(\mathbf X_k)
+\langle \mathbf M_k,\mathbf X_{k+1}-\mathbf X_k\rangle
+\langle\nabla f(\mathbf X_k)-\mathbf M_k,
\mathbf X_{k+1}-\mathbf X_k\rangle
+\frac{\eta^2}{2}
\\
&\le f(\mathbf X_k)
-\eta\|\mathbf M_k\|_{{\bf v},{\bf q}}
+\eta\|\nabla f(\mathbf X_k)-\mathbf M_k\|_{{\bf v},{\bf q}}
+\frac{\eta^2}{2}
\\
&\le f(\mathbf X_k)
-\eta\|\nabla f(\mathbf X_k)\|_{{\bf v},{\bf q}}
+2\eta\|\Delta \mathbf M_k\|_{{\bf v},{\bf q}}
+\frac{\eta^2}{2},
\end{align*}
where the last two inequalities use the weighted primal-dual H\"older inequality
and the triangle inequality.
\end{proof}

\begin{proof}[\textbf{Proof of Lemma \ref{lem:multi-key-relation}}]
The statement is trivial for $k=0$. For $k\ge1$, \eqref{update-mk-multi-1} and the same arguments used to derive \eqref{eq:DeltaMk-idtt-alg1} give
\begin{align*}
\Delta \mathbf M_k
=(1-\theta)^k\Delta \mathbf M_0
+\theta\sum_{t=0}^{k-1}(1-\theta)^t\Delta \mathbf G_{k-t}
-\sum_{t=0}^{k-1}(1-\theta)^{t+1}
\big(\nabla f(\mathbf X_{k-t})-\nabla f(\mathbf X_{k-t-1})\big).
\end{align*}
Taking the $\|\cdot\|_{{\bf v},{\bf q}}$ norm, using the triangle inequality,
Assumption \ref{asp:multi-basic}(c), and
$\|\mathbf X_{s+1}-\mathbf X_s\|_{{\bf w},{\bf p}}=\eta$ yields
\eqref{ineq:akeyrelation-multi}.
\end{proof}

\begin{proof}[\textbf{Proof of Lemma \ref{lem:multi-F-con}}]
Fix $k\ge1$. For each $j$, set
\[
\Lambda_{j,\varepsilon}
:=
(V_j V_j^T+\varepsilon I)^{1-\frac{q_j}{2}}
\succ0.
\]
By Lemma \ref{lem:multi-block-martingale} and
Assumption \ref{asp:multi-basic}(d), we obtain
\begin{align*}
\E\Big[
\Big\|
\sum_{t=0}^{k-1}(1-\theta)^t\Delta \mathbf G_{k-t}
\Big\|_{{\bf v},{\bf q}}^2
\Big]
&\le
\Big[
\sum_{j=1}^J
\Big(
\frac{L_j^{-1/2}}{\sqrt B}
\|\Lambda_{j,\varepsilon}\|_{S_{\frac{q_j}{2-q_j}}}^{1/2}
\Big[
\operatorname{Tr}\!\Big(
V_j V_j^T\Lambda_{j,\varepsilon}^{-1}
\Big)
\sum_{t=0}^{k-1}(1-\theta)^{2t}
\Big]^{1/2}
\Big)^\beta
\Big]^{\frac 2\beta}.
\end{align*}
Letting $\varepsilon\downarrow0$ yields
\begin{align*}
&\E\Big[
\Big\|
\sum_{t=0}^{k-1}(1-\theta)^t\Delta \mathbf G_{k-t}
\Big\|_{{\bf v},{\bf q}}^2
\Big]
\le \sigma_{{\bf v},{\bf q},B}^2
\sum_{t=0}^{k-1}(1-\theta)^{2t}
\le
\frac{\sigma_{{\bf v},{\bf q},B}^2}{\theta}.
\end{align*}
By Jensen's inequality,
\[
\E\Big[
\Big\|
\sum_{t=0}^{k-1}(1-\theta)^t\Delta \mathbf G_{k-t}
\Big\|_{{\bf v},{\bf q}}
\Big]
\le
\frac{\sigma_{{\bf v},{\bf q},B}}{\sqrt\theta}.
\]
Since $\mathbf M_0=\bar{\mathbf G}_0$, we have
$\Delta \mathbf M_0=\Delta \mathbf G_0$. Applying the same argument to this
one-term martingale difference sequence gives
$\E[\|\Delta \mathbf M_0\|_{{\bf v},{\bf q}}]
\le\sigma_{{\bf v},{\bf q},B}$.
Taking expectation in \eqref{ineq:akeyrelation-multi} gives
\begin{align*}
\E[\|\Delta \mathbf M_k\|_{{\bf v},{\bf q}}]
&\le
(1-\theta)^k\sigma_{{\bf v},{\bf q},B}
+\sqrt\theta\,\sigma_{{\bf v},{\bf q},B}
+\frac{\eta}{\theta}.
\end{align*}
Summing this inequality over $k=0,\ldots,K-1$ proves
\eqref{ineq:expec-upbd-multi-norm}.
\end{proof}

\begin{proof}[\textbf{Proof of Theorem \ref{thm:multi-F-con}}]
Taking expectation in \eqref{ineq:matrix-sign-descent-multi}, telescoping over
$k=0,\ldots,K-1$, and using \eqref{ineq:expec-upbd-multi-norm}, we have
\begin{align*}
\frac{1}{K}
\sum_{k=0}^{K-1}
\E[\|\nabla f(\mathbf X_k)\|_{{\bf v},{\bf q}}]
&\le
\frac{\Delta_f}{\eta K}
+\frac{\eta}{2}
+\frac{2}{K}
\sum_{k=0}^{K-1}
\E[\|\Delta \mathbf M_k\|_{{\bf v},{\bf q}}]
\\
&\le
\frac{\Delta_f}{\eta K}
+\frac{\eta}{2}
+\frac{2\sigma_{{\bf v},{\bf q},B}}{\theta K}
+2\sqrt\theta\,\sigma_{{\bf v},{\bf q},B}
+\frac{2\eta}{\theta}
\\
&\le
\frac{\Delta_f}{\eta K}
+\frac{2\sigma_{{\bf v},{\bf q},B}}{\theta K}
+2\sqrt\theta\,\sigma_{{\bf v},{\bf q},B}
+\frac{3\eta}{\theta}.
\end{align*}
Using the definition of $\eta$ in \eqref{def:eta-theta-F-multi}, this becomes
\[
\frac{1}{K}
\sum_{k=0}^{K-1}
\E[\|\nabla f(\mathbf X_k)\|_{{\bf v},{\bf q}}]
\le
4\sqrt{\frac{\Delta_f}{\theta K}}
+2\sqrt\theta\,\sigma_{{\bf v},{\bf q},B}
+\frac{2\sigma_{{\bf v},{\bf q},B}}{\theta K}.
\]
If $\theta=1$, then
$\sigma_{{\bf v},{\bf q},B}\le\sqrt{\Delta_f/K}$, and hence the RHS of the previous inequality
is bounded by $8\sqrt{\Delta_f/K}$. If $\theta<1$, the definition of
$\theta$ gives
\[
\frac{1}{K}
\sum_{k=0}^{K-1}
\E[\|\nabla f(\mathbf X_k)\|_{{\bf v},{\bf q}}]
\le
6\sqrt{\sigma_{{\bf v},{\bf q},B}}
\left(\frac{\Delta_f}{K}\right)^{1/4}
+\frac{2\sigma_{{\bf v},{\bf q},B}^2}{\sqrt{\Delta_f K}}.
\]
Combining the two cases yields \eqref{ineq:ave-stat-F-multi}.
\end{proof}

\subsubsection{Proof of the main results in Section \ref{sec:multi-reg}}\label{sec:pf-multi-reg}

\begin{proof}[\textbf{Proof of Lemma \ref{lem:multi-qo-descent}}]
By \eqref{eq:update-xk-multi-2} and the duality of
$\|\cdot\|_{{\bf w},{\bf p}}$ and $\|\cdot\|_{{\bf v},{\bf q}}$,
\[
\|\mathbf X_{k+1}-\mathbf X_k\|_{{\bf w},{\bf p}}
=\eta\|\mathbf M_k\|_{{\bf v},{\bf q}},
\qquad
\langle\mathbf M_k,\mathbf X_{k+1}-\mathbf X_k\rangle
=-\eta\|\mathbf M_k\|_{{\bf v},{\bf q}}^2.
\]
Using \eqref{ineq:desc-multi} with
$(\mathbf Y,\mathbf X)=(\mathbf X_{k+1},\mathbf X_k)$, we obtain
\begin{align*}
f(\mathbf X_{k+1})
&\le f(\mathbf X_k)
+\langle\nabla f(\mathbf X_k),\mathbf X_{k+1}-\mathbf X_k\rangle
+\frac{1}{2}\|\mathbf X_{k+1}-\mathbf X_k\|_{{\bf w},{\bf p}}^2\\
&= f(\mathbf X_k)
+\langle\mathbf M_k,\mathbf X_{k+1}-\mathbf X_k\rangle
+\langle\nabla f(\mathbf X_k)-\mathbf M_k,\mathbf X_{k+1}-\mathbf X_k\rangle +\frac{1}{2}\|\mathbf X_{k+1}-\mathbf X_k\|_{{\bf w},{\bf p}}^2\\
&\le f(\mathbf X_k)
-\eta\|\mathbf M_k\|_{{\bf v},{\bf q}}^2
+\frac{\eta}{2}\|\nabla f(\mathbf X_k)-\mathbf M_k\|_{{\bf v},{\bf q}}^2\\
&\hspace{3cm}
+\frac{1}{2\eta}\|\mathbf X_{k+1}-\mathbf X_k\|_{{\bf w},{\bf p}}^2
+\frac{1}{2}\|\mathbf X_{k+1}-\mathbf X_k\|_{{\bf w},{\bf p}}^2\\
&= f(\mathbf X_k)
-\left(\frac{1}{2\eta}-\frac{1}{2}\right)
\|\mathbf X_{k+1}-\mathbf X_k\|_{{\bf w},{\bf p}}^2
+\frac{\eta}{2}\|\Delta \mathbf M_k\|_{{\bf v},{\bf q}}^2\\
&= f(\mathbf X_k)
-\left(\frac{1}{4\eta}-\frac{1}{4}\right)
\|\mathbf X_{k+1}-\mathbf X_k\|_{{\bf w},{\bf p}}^2
-\frac{\eta}{4}(1-\eta)\|\mathbf M_k\|_{{\bf v},{\bf q}}^2
+\frac{\eta}{2}\|\Delta \mathbf M_k\|_{{\bf v},{\bf q}}^2\\
&\le f(\mathbf X_k)
-\left(\frac{1}{4\eta}-\frac{1}{4}\right)
\|\mathbf X_{k+1}-\mathbf X_k\|_{{\bf w},{\bf p}}^2
-\frac{\eta}{8}(1-\eta)
\|\nabla f(\mathbf X_k)\|_{{\bf v},{\bf q}}^2\\
&\hspace{3cm}
+\frac{\eta}{4}(3-\eta)
\|\Delta \mathbf M_k\|_{{\bf v},{\bf q}}^2,
\end{align*}
where we used the weighted H\"older inequality and
$-\|a-b\|^2\le -\frac12\|a\|^2+\|b\|^2$ in the last step.
\end{proof}

\begin{proof}[\textbf{Proof of Lemma \ref{lem:multi-square-deter}}]
The case $k=0$ is trivial. For $k\ge1$, \eqref{update-mk-multi-2} and the same arguments used to derive \eqref{eq:DeltaMk-idtt-alg1} give
\[
\Delta \mathbf M_k
=(1-\theta)^k\Delta \mathbf M_0
+\theta\sum_{t=0}^{k-1}(1-\theta)^t\Delta \mathbf G_{k-t}
-(1-\theta)\sum_{t=0}^{k-1}(1-\theta)^t
\big(\nabla f(\mathbf X_{k-t})-\nabla f(\mathbf X_{k-t-1})\big).
\]
Taking the $\|\cdot\|_{{\bf v},{\bf q}}$ norm and using
Assumption \ref{asp:multi-basic}(c) yields
\begin{align*}
\|\Delta \mathbf M_k\|_{{\bf v},{\bf q}}
&\le
(1-\theta)^k\|\Delta \mathbf M_0\|_{{\bf v},{\bf q}}
+\theta
\Big\|\sum_{t=0}^{k-1}(1-\theta)^t\Delta \mathbf G_{k-t}\Big\|_{{\bf v},{\bf q}}\\
&\quad
+(1-\theta)
\sum_{t=0}^{k-1}(1-\theta)^t
\|\mathbf X_{k-t-1}-\mathbf X_{k-t}\|_{{\bf w},{\bf p}}.
\end{align*}
Squaring both sides and applying Jensen's inequality to the last weighted sum,
\[
\left(
\sum_{t=0}^{k-1}(1-\theta)^t
\|\mathbf X_{k-t-1}-\mathbf X_{k-t}\|_{{\bf w},{\bf p}}
\right)^2
\le
\frac{1}{\theta}
\sum_{t=0}^{k-1}(1-\theta)^t
\|\mathbf X_{k-t-1}-\mathbf X_{k-t}\|_{{\bf w},{\bf p}}^2,
\]
proves \eqref{ineq:square-multi-deter}.
\end{proof}

\begin{proof}[\textbf{Proof of Lemma \ref{lem:multi-qo-error}}]
As in the proof of Lemma \ref{lem:multi-F-con}, Lemma \ref{lem:multi-block-martingale} 
with
$\Lambda_{j,\varepsilon}
=(V_j V_j^T+\varepsilon I)^{1-\frac{q_j}{2}}$
and $\varepsilon\downarrow0$ gives
\begin{align*}
\E\Big[
\Big\|\sum_{t=0}^{k-1}(1-\theta)^t
\Delta \mathbf G_{k-t}\Big\|_{{\bf v},{\bf q}}^2
\Big]
&\le
\sigma_{{\bf v},{\bf q},B}^2
\sum_{t=0}^{k-1}(1-\theta)^{2t}
\le
\frac{\sigma_{{\bf v},{\bf q},B}^2}{\theta}.
\end{align*}
Moreover, $\mathbf M_0=\bar{\mathbf G}_0$ implies
$\E[\|\Delta \mathbf M_0\|_{{\bf v},{\bf q}}^2]
\le \sigma_{{\bf v},{\bf q},B}^2$.
Taking expectation in \eqref{ineq:square-multi-deter}, we obtain
\begin{align*}
\E[\|\Delta \mathbf M_k\|_{{\bf v},{\bf q}}^2]
&\le
3(1-\theta)^{2k}\sigma_{{\bf v},{\bf q},B}^2
+3\theta\sigma_{{\bf v},{\bf q},B}^2
+\frac{3(1-\theta)^2}{\theta}
\sum_{t=0}^{k-1}(1-\theta)^t
\E[\|\mathbf X_{k-t-1}-\mathbf X_{k-t}\|_{{\bf w},{\bf p}}^2].
\end{align*}
Summing over $k=0,\ldots,K-1$ yields
\eqref{ineq:expec-upbd-spec-norm-multi}.
\end{proof}

\begin{proof}[\textbf{Proof of Theorem \ref{thm:multi-spec-con}}]
Taking expectation in \eqref{ineq:matrix-sign-descent-qo-multi}, telescoping
over $k=0,\ldots,K-1$, and using
\eqref{ineq:expec-upbd-spec-norm-multi}, we have
\begin{align*}
&\frac{1-\eta}{8K}
\sum_{k=0}^{K-1}
\E[\|\nabla f(\mathbf X_k)\|_{{\bf v},{\bf q}}^2]
\\
\le&\frac{\Delta_f}{K\eta}
-\frac{1}{K\eta}\left(\frac{1}{4\eta}-\frac{1}{4}\right)
\sum_{k=0}^{K-1}
\E[\|\mathbf X_{k+1}-\mathbf X_k\|_{{\bf w},{\bf p}}^2]
+\frac{3-\eta}{4K}
\sum_{k=0}^{K-1}
\E[\|\Delta \mathbf M_k\|_{{\bf v},{\bf q}}^2]
\\
\le&\frac{\Delta_f}{K\eta}
-\left[
\frac{1}{K\eta}\left(\frac{1}{4\eta}-\frac{1}{4}\right)
-\frac{3}{\theta^2}\cdot\frac{3-\eta}{4K}
\right]
\sum_{k=0}^{K-1}
\E[\|\mathbf X_{k+1}-\mathbf X_k\|_{{\bf w},{\bf p}}^2]\\
&\quad
+\frac{3(3-\eta)}{4K\theta}\sigma_{{\bf v},{\bf q},B}^2
+\frac{3(3-\eta)\theta}{4}\sigma_{{\bf v},{\bf q},B}^2.
\end{align*}
With $\eta=\theta/4$ and $\theta\le1$, the coefficient of the second term is
nonnegative and $1-\eta\ge3/4$. Therefore,
\begin{align*}
\frac{1}{K}\sum_{k=0}^{K-1}
\E[\|\nabla f(\mathbf X_k)\|_{{\bf v},{\bf q}}^2]
&\le
\frac{128\Delta_f}{3\theta K}
+\frac{24\sigma_{{\bf v},{\bf q},B}^2}{\theta K}
+24\theta\sigma_{{\bf v},{\bf q},B}^2\\
&\le
\frac{43(\Delta_f+\sigma_{{\bf v},{\bf q},B}^2)}{\theta K}
+24\theta\sigma_{{\bf v},{\bf q},B}^2.
\end{align*}
If $\theta<1$, the definition of $\theta$ gives
\[
\frac{1}{K}\sum_{k=0}^{K-1}
\E[\|\nabla f(\mathbf X_k)\|_{{\bf v},{\bf q}}^2]
\le
67\sqrt{
\frac{\sigma_{{\bf v},{\bf q},B}^2
(\Delta_f+\sigma_{{\bf v},{\bf q},B}^2)}{K}}.
\]
If $\theta=1$, then
$\sigma_{{\bf v},{\bf q},B}^2
\le(\Delta_f+\sigma_{{\bf v},{\bf q},B}^2)/K$, and hence
\[
\frac{1}{K}\sum_{k=0}^{K-1}
\E[\|\nabla f(\mathbf X_k)\|_{{\bf v},{\bf q}}^2]
\le
67\frac{\Delta_f+\sigma_{{\bf v},{\bf q},B}^2}{K}.
\]
Combining the two cases yields \eqref{ineq:ave-stat-spec-multi}.
\end{proof}







\appendix
\section*{Appendix}

\section{Related works}\label{apx:related-work}

In this part, we provide an overview of the relevant literature on the development of adaptive optimization methods. 


Driven by the transformative impact of neural network applications, modern architectures have evolved into increasingly complex compositions of heterogeneous building blocks \cite{gu2023mamba,he2016deep,shazeer2017outrageously,vaswani2017attention}. This rapid architectural progress has brought new optimization challenges to the forefront of deep learning research. Empirical evidence suggests that basic stochastic optimization methods such as SGD are often insufficient for training modern neural networks effectively, whereas adaptive optimization methods have become indispensable for modern deep learning due to the increasing complexity of contemporary architectures; see, e.g., \cite{zhang2020adaptive,zhang2024transformers}. As a result, developing principles for adaptive optimization methods has become a central theme in deep learning optimization research.

Earlier efforts developed coordinate-wise learning rates by viewing neural network training as stochastic optimization in a high-dimensional Euclidean space $\mathbb{R}^n$. Unlike SGD, which uses a single global learning rate, adaptive methods rescale each coordinate using historical gradient information. Representative examples include AdaGrad \cite{duchi2011adaptive}, RMSProp \cite{hinton2012rmsprop}, AdaDelta \cite{zeiler2012adadelta}, and Adam \cite{kingma2014adam}. Among them, Adam has become the most widely used optimizer in modern deep learning by combining momentum with coordinate-wise normalization based on second-moment estimates. Its theoretical properties have been extensively studied, see, e.g., \cite{chen2018on,defossez2022a,jiang2025stochastic,reddi2018convergence,zhang2022adam}, and numerous variants have been proposed, see, e.g., \cite{Liu2020On,loshchilov2019decoupled,luo2018adaptive,luo2024badam,zhang2025adam,zhuang2020adabelief}.

Despite the tremendous success of coordinate-wise learning-rate methods, the pursuit of more effective deep learning optimization has motivated researchers to explore alternative principles for designing adaptive optimization methods. One such perspective views optimizer design as controlling the geometry of the local search through the choice of norms \cite{bernstein2024old,large2024scalable}. 
Two representative methods, signSGD \cite{bernstein2018signsgd} and Muon \cite{jordan2024muon}, have attracted significant attention in modern deep learning optimization. In particular, Muon differs fundamentally from coordinate-wise adaptive methods, as it adapts to the spectrum of the update rather than to individual coordinates. This norm-based perspective has subsequently inspired a growing body of work on new algorithmic variants \cite{du2026newton,he2025low,he2025demuon,lau2025polargrad,liu2025muon,riabinin2025gluon} and their theoretical understanding \cite{chen2026muon,davis2025spectral,kovalev2025understanding,li2025note,shen2025convergence,shulgin2025beyond,su2025isotropic,yu2026sign}.




\bibliographystyle{abbrv}
\bibliography{ref}

\end{document}